\documentclass[10pts]{amsart}

\usepackage{amssymb}
\usepackage{graphicx}
\usepackage{amsfonts}
\usepackage{amsmath}
\usepackage{amsthm}
\usepackage{fancyhdr}
\usepackage{indentfirst}
\usepackage{hyperref}
\usepackage{comment}
\usepackage{color}
\usepackage{subcaption}
\usepackage{epsfig}
\usepackage{pst-grad} 
\usepackage{pst-plot} 
\usepackage[space]{grffile} 
\usepackage{etoolbox} 
\usepackage{mathrsfs} 

\newtheorem{theorem}{Theorem}[section]
\newtheorem{lemma}[theorem]{Lemma}

\theoremstyle{definition}
\newtheorem{definition}[theorem]{Definition}

\newtheorem{corollary}[theorem]{Corollary}

\newtheorem{proposition}[theorem]{Proposition}

\newtheorem{notation}[theorem]{Notation}
\newtheorem{convention}[theorem]{Convention}

\theoremstyle{remark}
\newtheorem{remark}[theorem]{Remark}

\numberwithin{equation}{section}

\newcommand{\R}{\mathbb{R}}

\newcommand{\ti}{\tilde}

\newcommand{\dist}{\operatorname{dist}}

\newcommand{\Tr}{\operatorname{Tr}}

\begin{document}

\title[Convex free-boundary mean curvature flow]{Contracting convex surfaces by mean curvature flow with free boundary on convex barriers
}
\date{\today}

\author[Sven Hirsch]{Sven Hirsch}
\address{Department of Mathematics, Duke University, Durham, NC 27708-0320, USA}
\email{sven.hirsch@duke.edu}

\author[Martin Li]{Martin Man-chun Li}
\address{Department of Mathematics, The Chinese University of Hong Kong, Shatin, N.T., Hong Kong}
\email{martinli@math.cuhk.edu.hk}

\begin{abstract}
We consider the mean curvature flow of compact convex surfaces in Euclidean $3$-space with free boundary lying on an arbitrary convex barrier surface with bounded geometry. When the initial surface is sufficiently convex, depending only on the geometry of the barrier, the flow contracts the surface to a point in finite time. Moreover, the solution is asymptotic to a shrinking half-sphere lying in a half space. This extends, in dimension two, the convergence result of Stahl for umbilic barriers to general convex barriers. We introduce a new perturbation argument to establish fundamental convexity and pinching estimates for the flow. Our result can be compared to a celebrated convergence theorem of Huisken for mean curvature flow of convex hypersurfaces in Riemannian manifolds.
\end{abstract}

\maketitle

\section{Introduction}
\label{S:intro}

Over the past few decades, geometric flows have blossomed and led to many striking applications in topology and geometry such as the proofs of Poincar\'{e} conjecture in three-dimensional topology by Hamilton \cite{Hamilton82} and Perelman \cite{Perelman1,Perelman2,Perelman3}, the Riemannian Penrose inequality in general relativity by Huisken-Ilmanen \cite{Huisken-Ilmanen01} and the Differentiable Sphere theorem by Brendle-Schoen \cite{Brendle-Schoen09} in Riemannian geometry. 
For all the results above, geometric flows are considered on manifolds and submanifolds without boundary, the behaviour of geometric flows for manifolds with boundary, on the other hand, is much less studied in the literature. 

It has been a longstanding question to define Ricci flow with boundary which is well-posed for general initial data. Recently, there has been some remarkable progress made by Gianniotis \cite{Gianniotis16a, Gianniotis16b}. Short-time existence and regularity were established under certain general geometric boundary conditions which are related to the boundary value problems for Einstein metrics posed by Anderson \cite{Anderson08, Anderson12}. It is an interesting direction to study the long-time behaviour of the flow.

For mean curvature flow, it is relatively easier to define the flow on submanifolds (especially hypersurfaces) with boundary. Two geometric boundary conditions have been most extensively studied. One is Dirichlet boundary condition where the motion of the boundary is prescribed (see for example \cite{White} and the references therein). The other one is Neumann boundary condition where the boundary contact angle is prescribed. When the contact angle is $\frac{\pi}{2}$, this is called Mean Curvature Flow (MCF) with free boundary and is the main object of study in this paper. The fundamental short-time existence and uniqueness for MCF with free boundary was first established by Stahl in \cite{Stahl96a}. The regularity and singularities of the flow were studied later for example, in \cite{Buckland05, Koeller12, Wheeler14} among many other. Certain weak formulations have been introduced in \cite{Giga-Sato93,Mizuno-Tonegawa15,Edelen18}. For mean-convex flow, substantial work has been done by Edelen \cite{Edelen16} and Edelen-Haslhofer-Ivaki-Zhu \cite{Edelen-Haslhofer-Ivaki-Zhu19} extending the foundational convexity estimates of Huisken-Sinestrari \cite{Huisken-Sinestrari99a,Huisken-Sinestrari99b} and regularity theory of White \cite{White00,White03,White15}. Special cases of MCF with free boundary were also studied, for example in the entire graphical case \cite{Wheeler14b}, in the Lorentzian setting \cite{Lambert14} and in the Lagrangian setting \cite{Evans-Lambert-Wood}.

One celebrated classical result of Huisken \cite{Huisken84} says that any convex hypersurfaces in $\mathbb{R}^{n+1}$ shrink to a round point in finite time under MCF. This result is later generalized to the Riemannian setting in \cite{Huisken86} provided that the initial hypersurface is convex enough to overcome the ambient geometry. In the free boundary setting, Stahl \cite{Stahl96b} prove that any convex hypersurface with free boundary lying on a flat hyperplane or a round hypersphere in $\mathbb{R}^{n+1}$ will shrink to a round point under the MCF with free boundary. A natural question is whether Stahl's convergence result can be extended to more general non-umbilic barrier surfaces. In this paper we answer this question affirmatively in dimension two (we refer the readers to Section \ref{S:prelim} for precise definitions).

\begin{theorem}
\label{T:main}
Let $S \subset \mathbb{R}^3$ be a complete, properly embedded oriented surface without boundary satisfying the following uniform bounds: there exist constants $K,L_1,L_2 \geq0$ such that
\begin{equation}
\label{A:Z-bound}
0 \leq \underline{Z}_S \leq \overline{Z}_S \leq K,
\end{equation}
where $\underline{Z}_S,\overline{Z}_S$ are the exterior and interior ball curvature respectively, and
\begin{equation}
\label{A:S-bound}
|\nabla_S A_S| \leq H_SL_1 \qquad \text{ and } \qquad  |\nabla_S^2 \mathring{A}_S| \leq L_2.
\end{equation}
Then there exists a constant $D\geq 0$, depending only on $K$, $L_1$ and $L_2$, such that the following holds: let $\Sigma_0$ be a compact connected surface smoothly immersed in $\mathbb{R}^3$ meeting $S$ orthogonally along its free boundary $\partial \Sigma_0 \subset S$, and suppose that on $\Sigma_0$ we have
\begin{equation}
\label{A:Convexity}
 h_{ij} > D g_{ij},
\end{equation}
then there exists a unique solution $\Sigma_t$ to the free-boundary mean curvature flow on a finite time interval $0 \leq t <T$ and the surfaces $\Sigma_t$ remains convex for all time. 
Furthermore, as $t \to T$, $\Sigma_t$ converges uniformly to half of a ``round point'' $p \in S$ in the sense that there is a sequence of rescalings which converge to a shrinking hemisphere with free boundary lying on a plane.
\end{theorem}

\begin{figure}[h]
\centering
\includegraphics[height=1.5in]{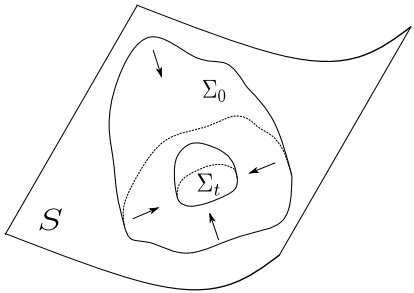}
\caption{A convex surface with free boundary contained in a convex barrier surface is evolving under mean curvature flow to a shrinking hemisphere.}
\label{F:1}
\end{figure}

\begin{remark} The assumptions (\ref{A:Z-bound}) and (\ref{A:S-bound}) are clearly satisfied for some $K,L_1,L_2$ for any compact convex barrier surface $S$. Furthermore, it will be apparent from the proof that the constant $D$ in (\ref{A:Convexity}) is close to zero if the barrier $S$ is close to a flat plane or a round sphere (in the $C^2$-sense). Therefore, we recover in particular the convergence result of Stahl in \cite{Stahl96b} for umbilic barriers.
\end{remark}

\begin{remark}
The geometry of the barrier surface $S$ can be thought of as an obstruction to convergence to a round point under the flow and the initial surface has to be sufficiently convex to overcome this obstruction. This can be compared to Huisken's result in \cite{Huisken86} where the obstruction arises from the geometry of the ambient Riemannian manifold. We expect that our results also hold in general Riemannian $3$-manifolds other than $\mathbb{R}^3$. For simplicity, we just present our result in the Euclidean case.
\end{remark}

\begin{remark}
Theorem \ref{T:main} also has the topological implication that any sufficiently convex free boundary surface $\Sigma_0$ is diffeomorphic to a disk. In fact, this also follows from Gauss-Bonnet as the induced metric on $\Sigma_0$ has positive Gauss curvature with convex boundary. If Theorem \ref{T:main} holds in higher dimensions, then it would be a non-trivial topological consequence of the flow. A version of the converse of the statement was established for spherical barrier by Ghomi-Xiong \cite{Ghomi-Xiong19}. It would be interesting to see if similar results hold for other convex barriers, provided that the surface is sufficiently positively curved (see \cite[Note 1.4]{Ghomi-Xiong19}).
\end{remark}

We would like to point out the differences with our main result in comparison with Huisken's convergence result \cite{Huisken86} in Riemannian manifolds. In \cite{Huisken86}, the surface has to be sufficiently convex depending on the zero-th and first order derivatives of the ambient curvature. The ambient space, when it is non-compact, is required to have a positive lower bound on the injectivity radius although his convergence result does not depend explicitly on this lower bound. In our main theorem, the convexity constant $D$ depends up to first order derivatives of the curvatures of the barrier surface $S$ as well as the second derivatives of the trace-free second fundamental form of $S$. Moreover, the ball curvature bounds in (\ref{A:Z-bound}) implies a positive lower bound on the boundary injectivity radius of $S$ and our convergence result depends explicitly on this bound. 

We now outline the main ideas of our proof of Theorem \ref{T:main}. As in many of the results for geometric flows, the major analytic tool is the maximum principle which first and second order conditions hold at any interior local minimum/maximum point. However, on (sub)-manifolds with boundary, the extrema can happen on the boundary at which we only get a first order inequality. This presents a major difficulty to deal with geometric flows on manifolds with boundary. In \cite{Stahl96b}, the barrier surface is totally umbilic, which can be exploited to avoid unwanted cross terms in the normal derivatives of the second fundamental form and hence the maximum principle can still be applied. However, if the barrier is not umblic, there are additional cross terms which are not controllable by lower order terms so the arguments in \cite{Stahl96b} are not sufficient.

To overcome these difficulties for general convex barriers, we use a perturbation argument of the second fundamental form which first appeared in \cite{Huisken-Sinestrari99b} (and more recently in \cite{Edelen16} and \cite{Edelen-Haslhofer-Ivaki-Zhu19}) by adding a suitably chosen perturbation tensor defined by
\begin{align*}
P^\Sigma(U,V):= & ( A_S(U,\nu) \nu_S^\flat(V) + A_S(V,\nu) \nu_S^\flat(U)) \; g_S(\nu,\nu)   \\
& -( g_S(U,\nu) \nu_S^\flat(V) + g_S(V,\nu) \nu_S^\flat(U) ) \; A_S(\nu,\nu)
\end{align*}
and where $A_S$ and $\nu_S$ are extended to $\R^3$ as explained in Section \ref{SS:barrier}. The perturbation tensor above kills off the cross terms of the second fundamental form along the boundary so that a simpler version of maximum principle \cite[Lemma 3.4]{Stahl96a} can be applied. Our perturbation tensor can be regarded as a refinement of the one used in \cite{Edelen16} that has better first order property along the boundary and moreover vanishes identically for the case of totally umbilic barriers. Various estimates for the perturbation tensor have to be done carefully so that the estimates depend only on the constants appearing in (\ref{A:Z-bound}) and (\ref{A:S-bound}).

Finally, we comment on the assumptions of Theorem \ref{T:main} regarding the dimension restriction and the convexity of the barrier surface. The only place where the dimension restriction comes in is to prove the convexity and pinching estimates in Section \ref{S:convexity} and \ref{S:pinching}. The boundary normal derivatives contain terms which can be combined in dimension two to give the mean curvature, whose positive lower bound is preserved under the flow when the barrier is convex. If we drop the convexity of the barrier surface, most of our arguments still go through (at places with a further perturbation argument) with a worse constant depending on $S$. For example, one can still prove that the mean curvature still blows up in finite time, provided that the surface is initially sufficiently convex. To keep our arguments relatively shorter and concise, we wish to address these issues in another forthcoming work.

The organization of the paper is as follows. In Section \ref{S:prelim}, we give a precise quantitative description of the barrier surface $S$ and a way to extend tensors from $S$ to all of $\mathbb{R}^3$. We also recall some fundamental facts about mean curvature flow with free boundary. In Section \ref{S:perturbed2ff}, we define the perturbation tensor and establish various foundation estimates which will be used crucially in the rest of our paper. In Section \ref{S:convexity}, we prove that convexity is preserved up to a fixed multiplicative factor provided that the surface is convex enough initially. A similar result was then established for the pinching of second fundamental form in Section \ref{S:pinching}. In Section \ref{S:pinching-A-circle}, we use Stampacchia iteration (generalized to the free boundary setting by \cite{Edelen16}) to prove the pinching estimate for the traceless second fundamental form. Finally, we establish the gradient bound for mean curvature in Section \ref{S:H-gradient-estimate}, from which our main result Theorem \ref{T:main} then follows.

\vspace{1ex}

\textbf{Acknowledgements.} 
This work was first carried out while the first author visited the Chinese University of Hong Kong in the summer of 2018. We appreciate the hospitality of the Mathematics Department there for providing a stimulating environment.
The authors would like to thank Prof. Richard Schoen and Prof. Shing-Tung Yau for their interest in this work. We also thank Simon Brendle, Binglong Chen, Yng-Ing Lee, Mao-Pei Tsui, Yuanlong Xin, Jonathan Zhu, Xiping Zhu for many useful comments and insightful discussions. M. L. is substantially supported by a research grant from the Research Grants Council of the Hong Kong Special Administrative Region, China [Project No.: CUHK 14323516] and CUHK Direct Grant [Project Code: 4053338].

\section{Preliminary results}
\label{S:prelim}

In this section, we give the precise definitions and notations that will be used for the rest of the paper. We begin with a quantitative description of the barrier surface and a way to extend tensors on $S$ to the entire $\mathbb{R}^3$. We then recall some basic facts about free-boundary mean curvature flow. Throughout this work, $\mathbb{R}^3$ is equipped with the Euclidean metric $\langle \cdot, \cdot \rangle$ with norm $|\cdot|$ and the flat connection $D$. We sometimes identify vectors and co-vectors when no ambiguity arises.

\subsection{The barrier surface}
\label{SS:barrier}

Throughout this paper, we let $S \subset \mathbb{R}^3$ be a properly embedded smooth surface (not necessarily compact nor connected) without boundary. We call $S$ the \emph{barrier surface} or simply the \emph{barrier}. Geometric quantities pertaining to the barrier surface $S$ will be indicated with a sub- or superscript, whichever is more convenient. Since $S$ is properly embedded, it is always orientable \cite{Samelson69} and we can fix a smooth global unit normal $\nu_S$. Recall from \cite[Definition 1]{Andrews-Langford-McCoy13} the notion of \emph{interior} and \emph{exterior ball curvature} at a point $p \in S$ defined respectively by
\begin{equation}
\label{E:int-ball-curv}
\overline{Z}_S(p):= \sup \left\{ \frac{2 \langle p-q,\nu_S(p) \rangle}{|p-q|^2} : q \in S, \; q \neq p \right\},
\end{equation}
\begin{equation}
\label{E:int-ball-curv}
\underline{Z}_S(p):= \inf \left\{ \frac{2 \langle p-q,\nu_S(p) \rangle}{|p-q|^2} : q \in S, \; q \neq p \right\}.
\end{equation}
Geometrically, $\overline{Z}_S(p)$ is the principal curvature of the largest ``interior ball'' (with respect to $\nu_S$) which touches $S$ at $p$ and $\underline{Z}_S(p)$ is such for the largest ``exterior ball''. We remark that $S$ does not necessarily bound a region in $\mathbb{R}^3$. The concept of ``interior'' and ``exterior'' is only defined locally relative to the ``outward'' unit normal $\nu_S$. 

With respect to the unit normal $\nu_S$, we define the second fundamental form of $S$ to be the symmetric $(0,2)$-tensor $A_S:TS \times TS \to \mathbb{R}$ where
\[ A_S(u,v):= - \langle D_u v, \nu_S \rangle\]
for any smooth tangential vector fields $u,v$ on $S$. The \emph{principal curvatures} are given by the eigenvalues of $A_S$ viewed as an endomorphism on $TS$. With our sign convention, $S$ is \emph{locally convex} if and only if $A_S$ is non-negative definite at every $p \in S$. Note that this is implied (see \cite[Proposition 4]{Andrews-Langford-McCoy13}) by the inequality $\underline{Z}_S(p) \geq 0$ at any $p \in S$. On the other hand, since a ball of curvature less than the largest principal curvature cannot touch $S$ from interior at $p$, we always have $\overline{Z}_S(p)$ at least as big as the largest principal curvature of $S$ at $p$. Therefore, the uniform bound (\ref{A:Z-bound}) implies that $S$ is a locally convex surface with principal curvatures at most $K$ everywhere. Note that $\overline{Z}_S$ and $\underline{Z}_S$ contain both information on the curvatures of $S$ (which is local) and the boundary injectivity radius \footnote{The boundary injectivity radius of $S$ is the maximal $\rho>0$ such that a $\rho$-tubular neighborhood of $S$ is diffeomorphic to $S \times (-\rho,\rho)$ under the normal exponential map of $S$.} of $S$ (which is non-local). For example, the slab region bounded by two parallel planes $S$ (appropriately oriented) which are of distance $r$ apart has zero principal curvatures but $\overline{Z}_S=2/r$.

In any local coordinates on $S$, we denote the components of $A_S$ by $\{h^S_{ij}\}$ and its covariant derivative $\nabla_S A_S$ by $\{ \nabla^S_k h^S_{ij}\}$. The induced metric on $S$ from $\mathbb{R}^3$ is denoted by $g_S$, which is a $(0,2)$-tensor on $S$ represented by $\{g^S_{ij}\}$ in local coordinates. We will use $g_S$ to raise or lower indices of tensors, e.g. $(h_S)^i_{\phantom{i}j}=g_S^{ik} h^S_{kj}$, adopting Einstein summation convention to sum over repeated indices. For any $p \geq 1$, $\nabla_S^p A_S$ denotes the $p$-th covariant derivative of $A_S$. Moreover, $|T|^2$ denotes the squared norm of any tensor, e.g. $|A_S|^2=h_S^{ij} h^S_{ij}$. We use $\mathring{A}_S$ to denote the trace-free second fundamental form of $S$ defined by
\begin{equation}
\label{E:A-o}
\mathring{A}_S:= A_S - \frac{1}{2} H_S g_S
\end{equation}
where $H_S:=\Tr A_S$ is the mean curvature of $S$.

For any $x \in \mathbb{R}^3$, we denote the minimal distance of $x$ to $S$ in $\mathbb{R}^3$ by $\dist (x,S)$. For any $\epsilon>0$, we denote the $\epsilon$-tubular neighborhood of $S$ by
\[ S_\epsilon:= \{x \in \mathbb{R}^3 \; : \; \dist(x,S) < \epsilon\}.\] 
From (\ref{A:Z-bound}) we know that for any $x \in S_{K^{-1}}$, there exists a unique $p_x \in S$ such that $\dist(x,S)=|x-p_x|$. Moreover, the (signed) \emph{distance function from $S$}, $d:S_{K^{-1}} \to \mathbb{R}$ defined by
\[ d(x):= \left\{ \begin{array}{cl}
-\dist(x,S) & \text{if $\langle x-p_x,\nu_S(p_x) \rangle \leq 0$}, \\
\dist(x,S) & \text{if $\langle x-p_x,\nu_S(p_x) \rangle \geq 0$}.
\end{array} \right.   \]
is a $C^\infty$ function on $S_{K^{-1}}$ satisfying the following at any $x \in S_{K^{-1}/2}$ (see \cite[Section 3]{Buckland05} and \cite{Gilbarg-Trudinger}):
\begin{equation}
\label{E:d-estimate}
Dd(x)=\nu_S(p_x), \quad D^2 d(x)(Dd(x),\cdot)=0 \quad \text{ and } \quad |D^2 d (x)| \leq 2 K. 
\end{equation}

Next, we want to extend $d$ to the whole $\mathbb{R}^3$ using a cut-off function. To this end, we fix a smooth cutoff function $\chi \in C^\infty(\mathbb{R})$ such that $\chi$ is decreasing with $\chi \equiv 1$ on $(-\infty,1)$, $\chi \equiv 0$ on $(2,+\infty)$, $\chi' \geq -2$ and $|\chi''| \leq 5$ everywhere. Using this cutoff function, we define the \emph{truncation function} $\chi_K:\mathbb{R}^3 \to \mathbb{R}$ by 
\[ \chi_K(x):= \chi \left( \frac{|d(x)|}{K^{-1}/4} \right). \]
Note that $\chi_K$ is a $C^\infty$ function on $\mathbb{R}^3$ which is supported in the tubular neighborhood $S_{K^{-1}/2}$, and $\chi_K \equiv 1$ in $S_{K^{-1}/4}$. Moreover, we have the following at any $x \in \mathbb{R}^3$:
\[ D \chi_K (x) = \left(4K \frac{d(x)}{|d(x)|} \chi' \right) Dd(x), \]
\[ D^2 \chi_K(x) =  \left(4K \frac{d(x)}{|d(x)|} \chi' \right) D^2d(x) + 16K^2 \chi'' Dd(x) \otimes Dd(x).\]
From above and (\ref{E:d-estimate}), we obtain easily the bounds $|D\chi_K| \leq 8K$ and $|D^2 \chi_K| \leq 96K^2$.

With the truncation function $\chi_K$ above, we now describe a general procedure to extend any $(0,q)$-tensor field $\phi$ on $S$ to the entire $\mathbb{R}^3$. At each $p \in S$, we first extend $\phi$, which is a $q$-linear form on $T_pS$, to a $q$-linear form on $T_p\mathbb{R}^3 \cong \mathbb{R}^3$ by defining $\phi(u_1,\cdots,u_q)=\phi(u_1^T, \cdots, u_q^T)$ where $(\cdot)^T$ denotes the orthogonal projection from $T_p \mathbb{R}^3$ onto $T_pS$. Then, we extend the $q$-linear form $\phi$ by parallel transport along normal geodesics emanating from $S$. Finally, we multiply $\phi$ by the truncation function $\chi_K$ so that it is a smooth $(0,q)$-tensor field globally defined on $\mathbb{R}^3$. By abuse of notation, we still denote the extended tensor field as $\phi$. Note that after the extension, $\phi$ is supported in the tubular neighborhood $S_{K^{-1}/2}$ and satisfies the bounds (curvatures of $S$ also appear when one differentiates $(\cdot)^T$)
\[ \|\phi\|_{C^0(\mathbb{R}^3)} \leq \|\phi\|_{C^0(S)}, \quad \| D \phi\|_{C^0(\mathbb{R}^3)} \leq 8K \|\phi\|_{C^0(S)}+ \|\nabla_S \phi\|_{C^0(S)},\]
\[  \|D^2 \phi\|_{C^0(\mathbb{R}^3)} \leq (96K^2+2L_1K) \|\phi\|_{C^0(S)}+16K \|\nabla_S \phi\|_{C^0(S)} + \|\nabla_S^2 \phi\|_{C^0(S)}.  \]

\begin{notation}
From now on, we use parenthesis to denote the dependence of constants. For example, $C(K)$ denotes any positive constant depending only on the constant $K$ in (\ref{A:Z-bound}). We use $C(S)$ denote any positive constant depending only on the constants $K$, $L_1$, $L_2$ in (\ref{A:Z-bound}) and (\ref{A:S-bound}). Moreover, we write $f=O(g)$ to mean that $|f| \leq C(S) |g|$.
\end{notation}

For example, when $k=0,1,2$, we have $\|\phi\|_{C^k(\mathbb{R}^3)} \leq C(K,L_1) \|\phi\|_{C^k(S)}$ for the extension of any $(0,q)$-tensor $\phi$ on $S$ to $\mathbb{R}^3$. By the same procedure, we can also extend any vector field, e.g. $\nu_S$, defined on $S$ to the entire $\mathbb{R}^3$ satisfying the same bounds.

\subsection{Free-boundary mean curvature flow}
\label{SS:FBMCF}

We now recall some preliminary results about free-boundary mean curvature flow. First, we restrict to the class of surfaces meeting (from inside) the barrier surface $S$ orthogonally.

\begin{definition}
\label{D:free-boundary}
Let $\Sigma$ denote a smooth two-dimensional surface with non-empty boundary $\partial \Sigma$. A \emph{free boundary surface} (with respect to $S$) is a smooth immersion $F:\Sigma \to \mathbb{R}^3$ such that
\[ F(\partial \Sigma) \subset S \quad \text{ and } \quad F_*N=\nu_S \circ F\]
where $N$ is the outward unit normal of $\partial \Sigma$ in $\Sigma$, with respect to $\Sigma$ equipped with the induced metric from the immersion $F$.
\end{definition}

Note that in case $S$ bounds a region $G$ in $\mathbb{R}^3$, from the definition above a free boundary surface does not have to lie completely either in $\overline{G}$ or $\overline{\mathbb{R}^3 \setminus G}$. The condition $F_*N=\nu_S \circ F$ means that along $F(\partial \Sigma) \subset S$, the surface $F(\Sigma)$ has to lie locally on one side of $S$ (as specified by the normal $\nu_S$). The surface $F(\Sigma)$, however, can intersect $S$ somewhere in its interior. 

We will assume throughout this paper that $\Sigma$ is orientable and we fix a choice of unit normal $\nu$ on $\Sigma$ with respect to the immersion $F:\Sigma \to \mathbb{R}^3$. We use $g$ and $A$ to denote respectively the induced metric and second fundamental form on $\Sigma$, where
\[ A(u,v):=  - \langle D_u v, \nu \rangle \]
for any tangential vector fields $u, v$ on $\Sigma$. The mean curvature of $\Sigma$ is denoted by $H:=\Tr A$. Similar to our previous discussion for the barrier surface, we denote, in any local coordinates of $\Sigma$, the components of $A$ by $\{h_{ij}\}$ and its covariant derivative $\nabla A$ by $\{ \nabla_k h_{ij}\}$. The induced metric on $\Sigma$ is denoted by $g$, whose components in local coordinates are given by $\{g_{ij}\}$. We will use $g$ to raise or lower indices of tensors. We write $\nabla$ and $\Delta$ for the intrinsic covariant derivative and Laplacian on $\Sigma$ respectively. Let $dV$ be the area form on $\Sigma$. There is a useful relationship between the second fundamental form $A$ of the free boundary surface $\Sigma$ and the second fundamental form $A_S$ of the barrier surface $S$ along the free boundary $\partial \Sigma$.

\begin{lemma}
\label{L:A-A^S}
Let $F:\Sigma \to \mathbb{R}^3$ be a free boundary surface with respect to $S$. For any $p \in \partial \Sigma$ and $X \in T_p \partial \Sigma$, we have $A(N,X)=-A^S(\nu\circ F,F_* X)$.
\end{lemma}

\begin{proof}
It follows easily by differentiating the identity $\langle \nu,\nu_S \circ F \rangle \equiv 0$ along $\partial \Sigma$ and using the free boundary condition $F_*N \equiv \nu_S \circ F$. See \cite[Proposition 4.5]{Edelen16} or \cite[Proposition 2.2]{Stahl96b}.
\end{proof}

We consider in this paper the mean curvature flow within the class of free boundary surfaces. It was first introduced by Huisken \cite{Huisken89} (in the graphical case) and Stahl \cite{Stahl96a}. Note that the definition in \cite{Stahl96a} does not require the surfaces to locally lie on one side of $S$ near their boundary.

\begin{definition}
Let $F_0:\Sigma \to \mathbb{R}^3$ be a free boundary surface as in Definition \ref{D:free-boundary}. We say that $F: \Sigma \times [0,T) \to \mathbb{R}^3$ is a solution to the \emph{free-boundary mean curvature flow} if for each $t \in [0,T)$, $F_t:=F(\cdot,t):\Sigma \to \mathbb{R}^3$ is a free boundary surface, $F(\cdot,0)=F_0$ and
\begin{equation}
\label{E:FBMCF}
\frac{\partial F}{\partial t} = -H \nu.
\end{equation}
By abuse of notation, we often write $\Sigma_t:=F_t(\Sigma)$.
\end{definition}

The fundamental short time existence and uniqueness for the free-boundary mean curvature flow was established by Stahl in \cite{Stahl96a}. For any smooth compact initial data $F_0:\Sigma \to \mathbb{R}^3$, there exist a unique solution to (\ref{E:FBMCF}) defined on a maximal time interval $[0,T)$. The solution is smooth for $t>0$ and $C^{2+\alpha,1+\alpha/2}$ up to $t=0$, with arbitrary $\alpha \in (0,1)$. Moreover, if $T<+\infty$, then $\sup_{\Sigma_t} |A| \to \infty$ as $t \to T$ \cite[Theorem 1.3]{Stahl96a}. It was shown recently by Guo \cite{Guo} that either $\sup_{\Sigma_t} |H| \to \infty$ or Length$(\partial \Sigma_t) \to \infty$ as $t \to T$, extending the remarkable work of Li and Wang \cite{Li-Wang19} to the free boundary setting.

We first recall the evolution equations for some basic geometric quantities on $\Sigma_t$. Note that we will suppress as usual the explicit dependence on $t$ for simplicity when no ambiguity arises. 

\begin{lemma}
\label{L:MCF}
Let $\{\Sigma_t\}_{t\in[0,T)}$ be a solution to the free-boundary mean curvature flow. Then, we have the following evolution equations for $t>0$,
\begin{itemize}
\item[(i)] $\partial_t g_{ij}=-2H h_{ij}$ 
\item[(ii)] $\partial_t \nu =\nabla H$
\item[(iii)] $\partial_t  dV=-H^2 dV$
\item[(iv)] $\left( \partial_t -\Delta \right) h_{ij}=-2 H h_{im}h^m_{\phantom{m}j}  +|A|^2 h_{ij}$.
\item[(v)] $\left(\partial_t  -\Delta \right) H=|A|^2 H$.
\item[(vi)] $\left(\partial_t -\Delta \right) |A|^2 = 2|A|^4 -2 |\nabla A|^2$.
\item[(vii)] $\left(\partial_t -\Delta \right) \left( |A|^2 - \frac{1}{2}H^2 \right) = 2|A|^2 \left( |A|^2 - \frac{1}{2}H^2 \right) -2 \left( |\nabla A|^2 -\frac{1}{2} |\nabla H|^2 \right)$.
\end{itemize}
\end{lemma}

\begin{proof}
See \cite[Section 3]{Huisken84}.
\end{proof}

Besides the evolution equations, we also need the boundary normal derivatives of various geometric quantities. We first recall the following fundamental result on the mean curvature, which holds for any positive time.

\begin{lemma}
\label{L:N-H}
Along $\partial \Sigma$, we have $N(H)=h^S_{\nu \nu} H$ for $t >0$
\end{lemma}

\begin{proof}
We obtain the desired formula by differentiating the free boundary condition $\langle N,\nu \rangle \equiv 0$ along $\partial \Sigma$ with respect to $t$ and using Lemma \ref{L:MCF} (ii). See, for example,  \cite[Proposition 4.3]{Edelen16} or \cite[Proposition 2.1]{Stahl96b} for details.
\end{proof}

Using the evolution equation and the boundary normal derivative of $H$, we obtain the following useful corollary by the maximum principle (c.f. \cite[Theorem 3.1 and 3.2]{Stahl96b}). Note that our barrier surface $S$ is locally convex (i.e. $h^S_{\nu \nu} \geq 0$) under assumption (\ref{A:Z-bound}).

\begin{corollary}
\label{C:H-preserve}
Any non-negative lower bound of $H$ is preserved under the flow, i.e. if $H \geq H_0 \geq 0$ at $t=0$ for some constant $H_0 \geq 0$, then $H \geq H_0$ for all $t>0$.
\end{corollary}

Note that when $H_0>0$, $H$ must in fact blow up in finite time with $T \leq H_0^{-2}$. Note that Lemma \ref{L:N-H} uses the evolution equation of $\nu$ under mean curvature flow and hence does not hold for a general free boundary surface $\Sigma$. 

The boundary normal derivatives of the second fundamental form $A$ were computed by \cite[Theorem 2.4]{Stahl96b} and \cite[Lemma 6.1]{Edelen16}. We recall their formula here, specializing to two-dimensional surfaces. 

\begin{convention}
At any point $p \in \partial \Sigma$, we always choose local Fermi coordinates in $\Sigma$ around $p$ so that along $\partial \Sigma$, $\partial_1 \equiv N$ and $\partial_2$ is a unit vector field tangent to $\partial \Sigma$. Moreover, the integral curves of $\partial_1$ are geodesics in $\Sigma$. We refer the readers to \cite[Section 2]{Marques05} for a more detailed discussion about Fermi coordinates.
\end{convention}

\begin{lemma}
\label{L:A-N}
At every $p \in \partial \Sigma$, we have for $t>0$
\begin{equation}
\label{E:N-h11}
\nabla_1 h_{11}=2h^S_{22}H +(h^S_{\nu \nu} -3 h^S_{22}) h_{11}+ \nabla_\nu^S h^S_{22}
\end{equation}
\begin{equation}
\label{E:N-h22}
\nabla_1 h_{22}= h^S_{22} H + (h^S_{\nu \nu} -3 h^S_{22}) h_{22} - \nabla_\nu^S h^S_{22} .
\end{equation}
\end{lemma}

\begin{proof}
It follows immediately from \cite[Lemma 6.1]{Edelen16} and that $H=h_{11}+h_{22}$, $H^S=h^S_{22}+h^S_{\nu \nu}$. Note that (\ref{E:N-h11}) uses the evolution equation in Lemma \ref{L:N-H}. On the other hand, (\ref{E:N-h22}) does not use any evolution equation and thus holds for any free boundary surface (without being a solution to the free-boundary mean curvature flow).
\end{proof}

From Lemma \ref{L:A-N} we see that the expression
\begin{align}
\label{E:N-A2} N (|A|^2) = & 6h^S_{22} H h_{11} +2 (h^S_{\nu\nu}-2 h^S_{22}) |A|^2 -4 h^S_{22}h^2_{11}  \\
& +2\nabla^S_\nu h^S_{22} (h_{11}-h_{22}) +4 h_{12} \nabla_1 h_{12} - 4 (h^S_{\nu\nu} -2 h^S_{22}) h^2_{12}  \nonumber
\end{align}
contains a term involving $\nabla_1 h_{12}$, which is not controllable. Note that when $S=\mathbb{S}^2$, the above formula simplifies to (note that $h_{12}=h^S_{2\nu}=0$)
\begin{equation}
\label{E:N-A^2}
N(|A|^2)=6 Hh_{11} -2|A|^2 -4h_{11}^2 =O(|A|^2)
\end{equation}
which implies $N |A| =O(|A|)$ and hence $N|A|=O(H)$ if $\Sigma$ is convex (as $|A| \leq H$). This observation is crucial in establishing the pinching estimate for $S=\mathbb{S}^2$ in \cite{Stahl96b}. Controlling the terms in (\ref{E:N-A2}) is the major difficulty to generalize Stahl's umbilic convergence result in \cite{Stahl96b} to general convex barrier surfaces. We will handle this by introducing a new perturbed second fundamental form with desired properties at $S$ up to first order.

\section{Perturbed second fundamental form}
\label{S:perturbed2ff}

In this section, we define our perturbation tensor which is the crucial new ingredient to deal with non-umbilic barriers. We carefully derive its basic properties and estimates which are required for later sections.

\subsection{The perturbation tensor}
\label{SS:perturbation}

We define an auxiliary $(0,5)$-tensor $P$ on $\mathbb{R}^3$ which is solely determined by the barrier surface $S$. Recall that $A_S$ and $g_S$ are symmetric $(0,2)$-tensors on $S$. By the extension procedure described in Section \ref{S:prelim}, we consider them as $(0,2)$-tensors defined on $\mathbb{R}^3$. On the other hand, at each $p \in S$, consider the co-vector $\nu_S^\flat$ dual to the vector $\nu_S$ at $p$ (i.e. $\nu_S^\flat(\cdot):=\langle \nu_S,\cdot \rangle$). We have then a $1$-form on $\mathbb{R}^3$ defined only along $S$. By a similar extension procedure as in Section \ref{SS:barrier} but without doing the tangential projection, we can regard $\nu_S^\flat$ as a $1$-form globally defined on $\mathbb{R}^3$ satisfying the following uniform bounds:
\[ |\nu^\flat_S| \leq 1, \; |D \nu^\flat_S| \leq 9K \quad \text{ and } \quad |D^2 \nu^\flat_S| \leq 104 K^2 +2 K L_1.\]

With these extensions understood, we make the following definition.

\begin{definition}
\label{D:P}
Let $P$ be the $(0,5)$-tensor on $\mathbb{R}^3$ defined by
\begin{align*}
P(U,V,X,Y,Z):= & ( A_S(U,X) \nu_S^\flat(V) + A_S(V,X) \nu_S^\flat(U)) \; g_S(Y,Z)   \\
& -( g_S(U,X) \nu_S^\flat(V) + g_S(V,X) \nu_S^\flat(U) ) \; A_S(Y,Z).  
\end{align*}
\end{definition}

By our way of extension, $P$ is clearly smooth and supported in the tubular neighborhood $S_{K^{-1}/2}$. One can also express $P$ in terms of the tracefree second fundamental form $\mathring{A}_S$ defined in (\ref{E:A-o}),
\begin{align*}
P(U,V,X,Y,Z)= & ( \mathring{A}_S(U,X)\nu_S^\flat(V)  + \mathring{A}_S(V,X)  \nu_S^\flat(U) ) \; g_S(Y,Z)   \\
& -( g_S(U,X)\nu_S^\flat(V) + g_S(V,X)  \nu_S^\flat(U) ) \; \mathring{A}_S(Y,Z)  
\end{align*}
From this expression it follows that $P$ vanishes identically whenever $S$ is \emph{totally umbilic} (i.e. $\mathring{A}_S \equiv 0$). It is clear from the definition that $P$ is symmetric in the first two slots, i.e. $P(U,V,X,Y,Z)=P(V,U,X,Y,Z)$. Moreover, we have the following estimates (note that $(D \nu_S^\flat)^T = A_S$ along $S$):
\[  \|P\|_{C^0(\mathbb{R}^3)} \leq 4 \|\mathring{A}_S\|_{C^0(S)}, \qquad \|DP\|_{C^0(\mathbb{R}^3)} \leq C(K) \|\mathring{A}_S\|_{C^1(S)}, \]
\[ \|D^2P\|_{C^0(\mathbb{R}^3)} \leq C(K,L_1) \|\mathring{A}_S\|_{C^2(S)}\]
Therefore, we have 
\begin{equation}
\label{E:P-C^2}
\|P\|_{C^2(\mathbb{R}^3)} \leq C(K,L_1) \|\mathring{A}_S\|_{C^2(S)} \leq C(S).
\end{equation}

There are some nice additional properties of $P$ which hold for points lying on the barrier surface $S$. 

\begin{lemma}
\label{L:P-0}
The following holds on $S$:
\begin{itemize}
\item[(i)] $P(U,V,X,Y,Z)=0$ whenever one of the $X$, $Y$ and $Z$ belongs to $(TS)^\perp$,
\item[(ii)] $P(U,V,X,Y,Z)=0$ whenever $U,V \in TS$,
\item[(iii)] $P(U,V,V,V,V)=0$ whenever $V \in TS$,
\item[(iv)] $P(\nu_S,\nu_S,X,Y,Z)=0$,
\item[(v)] $D_{\nu_S} P=0$.
\end{itemize}
\end{lemma}

\begin{proof}
(i) - (iv) follow directly from the definition of $P$ and (v) follows from the way we extend the tensor fields from $S$ to $\mathbb{R}^3$.
\end{proof}

\begin{definition}
\label{D:P-Sigma}
Given any free boundary surface $\Sigma$ with unit normal $\nu$, we define a symmetric $(0,2)$-tensor $P^\Sigma:T\Sigma \times T\Sigma \to \mathbb{R}$ on $\Sigma$ by
\[ P^\Sigma(u,v):= P(u,v,\nu,\nu,\nu), \]
where $P$ is the $(0,5)$-tensor defined on $\mathbb{R}^3$ as in Definition \ref{D:P}.
\end{definition}

Note that as $g^S(\nu,\nu)=1$ and $g^S(u,\nu)=0$ for all $u \in T\Sigma$ along $\partial \Sigma$, our perturbation term reduces to the one considered in \cite[Definition 4.5.1]{Edelen16}: for any $u,v \in T_p\Sigma$ where $p \in \partial \Sigma \subset S$, we have
\begin{equation}
\label{E:P-Sigma}
P^\Sigma(u,v)= A_S(u,\nu) \langle v,\nu_S \rangle +A_S(v,\nu) \langle u, \nu_S \rangle. 
\end{equation}

\begin{lemma}
\label{L:P-def}
Along $\partial \Sigma$, we have
\[ P^\Sigma_{11}=P^\Sigma_{22}=0 \quad \text{ and } \quad P^\Sigma_{12}=-h_{12}.\]
\end{lemma}

\begin{proof}
It follows easily from (\ref{E:P-Sigma}) and Lemma \ref{L:A-A^S}.
\end{proof}

The perturbation term (\ref{E:P-Sigma}) is already enough for the purpose of proving the convexity estimates in \cite{Edelen16}. However, for our purpose we need a stronger condition at the boundary, which is given by the following lemma.

\begin{lemma}
\label{L:P-def-1}
Along $\partial \Sigma$, we have
\[ \nabla_1 P^\Sigma_{11}=\nabla_1 P^\Sigma_{22}=0.\]
\end{lemma}

\begin{proof}
By the definition of Fermi coordinates along $\partial \Sigma$, we have 
\[ \nabla_1 \partial_1=0 \quad \text{ and } \quad \nabla_1 \partial_2= h^S_{22} \partial_2.\]
Combining this with the Weingarten equations, we have
\[ D_{\partial_1} \partial_1 = -h_{11} \nu, \quad  D_{\partial_1} \partial_2 = h^S_{22} \partial_2 -h_{12} \nu \quad \text{ and } \quad D_{\partial_1} \nu = h_{11} \partial_1 + h_{12} \partial_2.\]
Therefore, we obtain
\begin{align*}
\nabla_1 P^\Sigma_{11} =& \partial_1P_{11\nu\nu\nu} \\
=& D_1 P_{11\nu\nu\nu} -2 h_{11} P_{\nu 1 \nu \nu \nu}+h_{11} (P_{111\nu\nu}+P_{11\nu 1 \nu}+P_{11\nu \nu 1}) \\
& +h_{12} (P_{112\nu \nu}+P_{11\nu 2 \nu}+P_{11 \nu \nu 2})
\end{align*}
which vanishes by Lemma \ref{L:P-0} (iii) (iv) (v) and that $\partial_1=\nu_S$, and $\nu \in TS$ along $\partial \Sigma$. Similarly,
\begin{align*}
\nabla_1 P^\Sigma_{22} =& D_1 P_{22 \nu \nu \nu}-2 h_{12} P_{\nu 2 \nu \nu \nu} +h_{11} (P_{221\nu\nu}+P_{22\nu 1 \nu}+P_{22\nu \nu 1}) \\
& +h_{12} (P_{222\nu \nu}+P_{22\nu 2 \nu}+P_{22 \nu \nu 2})
\end{align*}
which vanishes by Lemma \ref{L:P-0} (ii) (v) since $\partial_2, \nu \in TS$ along $\partial \Sigma$. 
\end{proof}

We derive now the evolution equation for the perturbation tensor $P^\Sigma$.

\begin{proposition}
\label{P:P-estimate}
Let $\{\Sigma_t\}_{t\in[0,T)}$ be a solution to the free-boundary mean curvature flow. Then, we have the following evolution equation:
\begin{align*}
(\partial_t -\Delta) P^\Sigma_{ij} =& 3|A|^2 P^\Sigma_{ij} +h_{pi} (h_{pk} P^\Sigma_{kj} -H P^\Sigma_{pj}) + h_{pj} (h_{pk} P^\Sigma_{ik} -H P^\Sigma_{ip}) \\
& +2 h_{pi} h_{pk} (P_{\nu jk\nu\nu} +P_{\nu j \nu k \nu}+P_{\nu j \nu \nu k}) \\
& +2 h_{pj} h_{pk} (P_{i\nu k \nu \nu} + P_{i \nu \nu k \nu} +P_{i \nu \nu \nu k}) \\
& -2 h_{pi} h_{pj} P_{\nu \nu \nu \nu \nu} -2 h_{p \ell}h_{pk} (P_{ijk\ell \nu} + P_{ijk\nu \ell} + P_{ij \nu k \ell}) \\
& - D^2_{p,p}  P_{ij\nu\nu\nu} + 2h_{pi} D_p P_{\nu j \nu \nu \nu} +2 h_{pj} D_p P_{i \nu \nu \nu \nu} \\
& -2 h_{pk} (D_p P_{ijk \nu \nu} +D_p P_{ij \nu k \nu} + D_p P_{ij \nu \nu k})
\end{align*}
Moreover, we have the following bounds
\begin{equation}
\label{E:P-bounds}
P^\Sigma= O(1), \quad \nabla P^\Sigma = O(1+|A|), \quad \nabla^2 P^\Sigma= O(1+|A|^2+|\nabla A|),
\end{equation}
\begin{equation}
\label{E:Delta-P-bounds}
(\partial_t -\Delta) P^\Sigma_{ij}=O(1+|A|^2).
\end{equation}
\end{proposition}

\begin{proof}
Choose any orthonormal geodesic coordinates $\partial_1,\partial_2$ centered at a point $x \in \Sigma$. Similar to the calculations in \cite[Proposition 5.1]{Edelen16}, we have
\[
\nabla_p P^\Sigma_{ij} = D_p P_{ij\nu \nu \nu} +h_{pk} (P_{ijk \nu \nu}+P_{ij \nu k \nu} +P_{ij \nu \nu k}) -h_{pi} P_{\nu j \nu \nu \nu} -h_{pj} P_{i \nu \nu \nu \nu}.
\]
This implies the bound $\nabla P^\Sigma = O(1+|A|)$. Differentiating once again, using Codazzi equation, we have
\begin{align*}
\nabla_q (D_p P_{ij \nu \nu \nu})=&D^2_{q,p} P_{ij \nu \nu \nu} - h_{qp} D_\nu P_{ij \nu\nu\nu} -h_{qi} D_p P_{\nu j \nu\nu\nu} -h_{qj} D_p P_{i \nu\nu\nu\nu} \\
& +h_{qk} (D_p P_{ijk \nu \nu} + D_p P_{ij\nu k\nu} +D_p P_{ij \nu \nu k}),
\end{align*}
\begin{align*}
\nabla_q(h_{pk} P_{ijk \nu \nu}) =& (\nabla_k h_{pq}) P_{ijk \nu \nu} +h_{pk} h_{q \ell} (P_{ijk \ell \nu} + P_{ijk \nu \ell}) \\
& +h_{pk} (D_q P_{ijk \nu\nu} -h_{qi} P_{\nu j k \nu\nu} -h_{qj} P_{i \nu k \nu\nu} -h_{qk} P_{ij \nu \nu \nu}),
\end{align*}
\begin{align*}
\nabla_q (h_{pi} P_{\nu j \nu \nu \nu}) =& (\nabla_i h_{pq}) P_{\nu j \nu\nu\nu} + h_{pi}(D_q P_{\nu j \nu \nu \nu} -h_{qj} P_{\nu \nu\nu\nu\nu}) \\
& +h_{pi} h_{qk}(P_{kj\nu\nu\nu} +P_{\nu jk\nu\nu} + P_{\nu j \nu k \nu} + P_{\nu j \nu \nu k}).
\end{align*}
This implies the bound $\nabla^2 P^\Sigma= O(1+|A|^2+|\nabla A|)$. Adding up the terms and summing over $p,q$, we have
\begin{align*}
\Delta P^\Sigma_{ij} =& -(\nabla_i H) P_{\nu j \nu\nu\nu} -(\nabla_j H) P_{i \nu \nu\nu\nu} +(\nabla_k H) (P_{ijk \nu\nu}+P_{ij \nu k \nu}+P_{ij \nu \nu k})\\
& -HD_\nu P_{ij \nu \nu \nu} - 3|A|^2 P^\Sigma_{ij} -h_{pi} h_{pk} P^\Sigma_{kj}  - h_{pj} h_{pk} P^\Sigma_{ik} \\
& -2 h_{pi} h_{pk} (P_{\nu jk\nu\nu} +P_{\nu j \nu k \nu}+P_{\nu j \nu \nu k}) \\
& -2 h_{pj} h_{pk} (P_{i\nu k \nu \nu} + P_{i \nu \nu k \nu} +P_{i \nu \nu \nu k}) \\
& +2 h_{pi} h_{pj} P_{\nu \nu \nu \nu \nu} +2 h_{p \ell}h_{pk} (P_{ijk\ell \nu} + P_{ijk\nu \ell} + P_{ij \nu k \ell}) \\
& + D^2_{p,p}  P_{ij\nu\nu\nu} - 2h_{pi} D_p P_{\nu j \nu \nu \nu} -2 h_{pj} D_p P_{i \nu \nu \nu \nu} \\
& +2 h_{pk} (D_p P_{ijk \nu \nu} +D_p P_{ij \nu k \nu} + D_p P_{ij \nu \nu k}).
\end{align*}
On the other hand, computing the time derivative gives
\begin{align*}
\partial_t P^\Sigma_{ij} =& -(\nabla_i H) P_{\nu j \nu\nu\nu} -(\nabla_j H) P_{i \nu \nu\nu\nu} +(\nabla_k H) (P_{ijk \nu\nu}+P_{ij \nu k \nu}+P_{ij \nu \nu k})\\
& -HD_\nu P_{ij \nu \nu \nu}-Hh_{ip} P^\Sigma_{pj} -H h_{jp} P^\Sigma_{ip}.
\end{align*}
Combining the last two equations yield the desired formula.
\end{proof}

\begin{remark}
Examining the proof carefully we have in fact the following
\[ |P^\Sigma | \leq 4 \|\mathring{A}_S\|_{C^0(S)}, \qquad |\nabla P^\Sigma | \leq C(K)\|\mathring{A}_S\|_{C^1(S)}  (1+|A|) ,\]
\[ |\nabla^2 P^\Sigma| \leq C(K,L_1) \|\mathring{A}_S\|_{C^2(S)} ( 1+|A|^2 +|\nabla A|),\]
\[ |(\partial_t -\Delta) P^\Sigma| \leq C(K,L_1) \| \mathring{A}_S\|_{C^2(S)} (1+|A|^2). \]
\end{remark}

\subsection{Perturbed second fundamental form}
\label{SS:perturbed-A}

We now use the perturbation tensor defined in the previous subsection to construct the new perturbed second fundamental form with desirable properties.

\begin{definition}
\label{D:perturbed-A}
Given a free boundary surface $\Sigma$ with unit normal $\nu$, the \emph{perturbed second fundamental form $\tilde{A}=(\tilde{h}_{ij})$ of $\Sigma$} is a symmetric $(0,2)$-tensor on $\Sigma$ defined by
\[ \tilde{A}(X,Y):= A(X,Y) + P^\Sigma (X,Y) \qquad \textrm{for all $X,Y \in T\Sigma$}.\]  
Moreover, we define the perturbed mean curvature to be $\tilde{H}:=\Tr \tilde{A}$.
\end{definition}

\begin{lemma}
\label{L:A-tilde-1}
Along $\partial \Sigma$, we have
\[ \tilde{h}_{11}=h_{11}, \; \tilde{h}_{22}=h_{22} \quad \text{ and } \quad \tilde{h}_{12}=0.\]
Hence, $|\tilde{A}| \leq |A|$ and $\tilde{H}=H$ along $\partial \Sigma$.
\end{lemma}

\begin{proof}
The statements follow directly from Lemma \ref{L:P-def}.
\end{proof}

Note that $\tilde{A}=A$ globally in $\Sigma$ when $S$ is totally umbilic as $P \equiv 0$. However, this is in general not true when $S$ is non-umbilic. 

We now compute the boundary normal derivatives for the perturbed second fundamental form.

\begin{lemma}
\label{L:N-htilde}
Along $\partial \Sigma$, we have
\begin{equation}
\label{E:N-h11-tilde}
\nabla_1 \tilde{h}_{11}=2h^S_{22}H +(h^S_{\nu \nu} -3 h^S_{22}) h_{11}+ \nabla_\nu^S h^S_{22}
\end{equation}
\begin{equation}
\label{E:N-h22-tilde}
\nabla_1 \tilde{h}_{22}= h^S_{22} H + (h^S_{\nu \nu} -3 h^S_{22}) h_{22} - \nabla_\nu^S h^S_{22} .
\end{equation}
Hence, we have $N\tilde{H} =h^S_{\nu\nu} \tilde{H}$.
\end{lemma}

\begin{proof}
It follows directly from Lemma \ref{L:A-N}, \ref{L:P-def-1} and \ref{L:A-tilde-1}. 
\end{proof}

\begin{lemma}
Along $\partial \Sigma$, we have
\begin{equation}
\label{E:N-|A|-tilde}
\frac{1}{2}N |\tilde{A}|^2 = 3h^S_{22}Hh_{11}+(h^S_{\nu\nu}-2h^S_{22})|\tilde{A}|^2-2h^S_{22} h^2_{11} + (\nabla^S_\nu h^S_{22}) (h_{11}-h_{22})
\end{equation}
In particular, we have the following inequality at any $|\tilde{A}|>0$,
\begin{equation}
\label{E:N-|A-tilde|-ineq}
N|\tilde{A}| \leq 3 h^S_{22} \frac{H}{|\tilde{A}|} h_{11} + (h^S_{\nu\nu}-2h^S_{22})|\tilde{A}| + \sqrt{2} |\nabla^S_\nu h^S_{22}|
\end{equation}
\end{lemma}

\begin{proof}
It follows by a straightforward calculation from Lemma \ref{L:N-htilde}. Note that we do not have a term (c.f. (\ref{E:N-A2})) involving $\nabla_1 \tilde{h}_{12}$ since $\tilde{h}_{12}=0$ along $\partial \Sigma$. Moreover, $|\tilde{A}|^2=h_{11}^2+h_{22}^2$ along $\partial \Sigma$ by Lemma \ref{L:A-tilde-1} and $h^S_{22} \geq 0$ by convexity of $S$.
\end{proof}

We now derive some bounds involving the evolution equation for the perturbed second fundamental form.

\begin{proposition}
\label{P:Atilde-evolution}
Let $\{\Sigma_t\}_{t\in[0,T)}$ be a solution to the free-boundary mean curvature flow. Then, we have the following bounds on the evolution equation:
\begin{equation}
\label{E:htilde-evolve}
(\partial_t -\Delta) \tilde{h}_{ij} = |A|^2 \tilde{h}_{ij} - 2H h_{im} \tilde{h}^m_{\phantom{m}j} +O(1+|A|^2),
\end{equation}
\begin{equation}
\label{E:Htilde-evolve}
(\partial_t -\Delta) \tilde{H} = |A|^2 \tilde{H} + O(1+|A|^2),
\end{equation}
\begin{equation}
\label{E:|Atilde|^2-evolve}
(\partial_t - \Delta) |\tilde{A}|^2 = 2|A|^2 |\tilde{A}|^2 - 2|\nabla \tilde{A}|^2 +O(1+|A|^2),
\end{equation}
\begin{align}
\label{E:pinch-evolve}
(\partial_t - \Delta) \left( |\tilde{A}|^2 - \frac{1}{2} \tilde{H}^2 \right) =& 2|A|^2 \left( |\tilde{A}|^2 - \frac{1}{2}\tilde{H}^2 \right) \\
& -2 \left( |\nabla \tilde{A}|^2 -\frac{1}{2} |\nabla \tilde{H}|^2 \right) + O(1+|A|^3). \nonumber
\end{align}
\end{proposition}

\begin{proof}
(\ref{E:htilde-evolve}) follows directly from Lemma \ref{L:MCF} (iv), Proposition \ref{P:P-estimate} and (\ref{E:P-C^2}). (\ref{E:Htilde-evolve}) then follows from (\ref{E:htilde-evolve}) together with Lemma \ref{L:MCF} (i) and (\ref{E:P-C^2}). For (\ref{E:|Atilde|^2-evolve}), we compute using Lemma \ref{L:MCF} (i) and (\ref{E:htilde-evolve})
\begin{align*}
\frac{1}{2} (\partial_t -\Delta) |\tilde{A}|^2 =& \frac{1}{2} \partial_t(g^{ik} g^{j\ell} \tilde{h}_{ij} \tilde{h}_{k \ell}) - \langle \Delta \tilde{A},\tilde{A} \rangle -|\nabla \tilde{A}|^2 \\
=& 2 H h^{ik} g^{j\ell} \tilde{h}_{ij} \tilde{h}_{k \ell} + g^{ik} g^{j \ell}  \tilde{h}_{k\ell}(\partial_t -\Delta) \tilde{h}_{ij} -|\nabla \tilde{A}|^2 \\
=& |A|^2 |\tilde{A}|^2 - |\nabla \tilde{A}|^2 +O(1+|A|^2).
\end{align*}
Finally, (\ref{E:pinch-evolve}) follows immediately from (\ref{E:Htilde-evolve}) and (\ref{E:|Atilde|^2-evolve}).
\end{proof}

Note that the error term in (\ref{E:|Atilde|^2-evolve}) is of order $|A|^2$ instead of $|A|^3$ (c.f. \cite[Theorem 5.3]{Edelen16}). On the other hand, we only get the error bound in the order of $|A|^3$ in (\ref{E:pinch-evolve}), which is enough for our purpose later. From (\ref{E:P-bounds}), we have the following bounds:
\begin{equation}
\label{E:A-Atilde-bound}
\tilde{A} = A+ O(1), \quad |\tilde{A}| = |A|+O(1),
\end{equation}
\begin{equation}
\label{E:nablaA-Atilde-bound}
\nabla \tilde{A} = \nabla A+ O(1+|A|), \quad |\nabla \tilde{A}| = |\nabla A|+O(1+|A|).
\end{equation}

\begin{corollary}
\label{C:|A-tilde|-evolve}
Whenever $|\tilde{A}| \geq 1$, we have 
\[ (\partial_t -\Delta) |\tilde{A}|  \leq |A|^2 |\tilde{A}| + O(|\tilde{A}|)\]
\end{corollary}

\begin{proof}
Note that
\[ (\partial_t-\Delta) |\tilde{A}| = \frac{1}{2} \frac{(\partial_t-\Delta)|\tilde{A}|^2}{|\tilde{A}|} + \frac{|\nabla|\tilde{A}||^2}{|\tilde{A}|} =|A|^2 |\tilde{A}| + \frac{|\nabla|\tilde{A}||^2- |\nabla \tilde{A}|^2}{|\tilde{A}|}+O(|\tilde{A}|)\]
from which the estimate follows from Kato's inequality that $|\nabla|\tilde{A}|| \leq |\nabla \tilde{A}|$.
\end{proof}

Note that we have the error term bounded by $|\tilde{A}|$ instead of $|\tilde{A}|^2$ as in \cite{Edelen16}. However it is also enough for our purpose to have the weaker bound.

\section{Preservation of convexity}
\label{S:convexity}

In this section we prove that convexity is preserved under free-boundary mean curvature flow, provided that the initial surface $\Sigma_0$ is convex \emph{enough} (depending only on $S$). When $S=\mathbb{S}^2$ or $\mathbb{R}^2$, this was established by Stahl in \cite[Theorem 4.4]{Stahl96b}. Our result generalizes this to \emph{arbitrary} convex barriers.

We first show that the any sufficiently large positive lower bound for the perturbed second fundamental form $\tilde{A}=(\tilde{h}_{ij})$ as defined in Definition \ref{D:perturbed-A} is preserved up to a fixed multiplicative factor. Our proof is based on a maximum principle argument applied to the symmetric $(0,2)$-tensor $\tilde{h}_{ij}$. The advantage of using the perturbed second fundamental form is that $\tilde{h}_{ij}$ decomposes at the boundary $\partial \Sigma$ by Lemma \ref{L:A-tilde-1}. Therefore, for the maximum principle arguments we only have to consider the boundary derivatives $\nabla_1 \tilde{h}_{11}$ and $\nabla_1 \tilde{h}_{22}$ but not the cross term $\nabla_1 \tilde{h}_{12}$, on which we have no control.

\begin{theorem}
\label{T:tilde-convexity-I}
There exists a constant $\tilde{D}_0=\tilde{D}_0(S)>0$ such that whenever $\{\Sigma_t\}_{t\in[0,T)}$ is a solution to the free-boundary mean curvature flow with 
\[ \tilde{h}_{ij} \geq \tilde{D} g_{ij} \quad \text{ at $t=0$}  \]
for some constant $\tilde{D} \geq \tilde{D}_0$, then we have
\[  \tilde{h}_{ij} > \frac{1}{2} \tilde{D} g_{ij} \quad \text{ for all $t\in [0,T)$} .\]
\end{theorem}

\begin{proof}
We argue by contradiction. Suppose there is a first time $t_0 \in (0,T)$ and a point $x_0 \in \Sigma$ such that $\tilde{h}(v,v)=\tilde{D}/2$ for some unit tangent vector $v \in T_{x_0}\Sigma_{t_0}$. We will derive a contradiction when $\tilde{D}>0$ is sufficiently large, depending only on $S$. There are two different cases to consider: either $x_0$ lies in the interior of $\Sigma$ or $x_0 \in \partial \Sigma$.

Suppose first $x_0$ is an interior point of $\Sigma$. We can extend $v$ to a neighborhood of $x_0$ in $\Sigma$ by parallel transport along radial geodesics (with respect to $\Sigma_{t_0}$) emanating from $x_0$, and then extend $v$ being constant in time. In other words, we have at $(x_0,t_0)$
\begin{equation}
\label{E:v-extend}
\nabla v =0 \quad \text{ and } \quad \partial_t v=0.
\end{equation}
Then the smooth function defined by
\[ f :=\tilde{h}(v,v) - \frac{1}{2} \tilde{D} g(v,v) \]
has an interior minimum at $(x_0,t_0)$ within a spacetime neighorbood of $(x_0,t_0)$ in $\Sigma \times (0,t_0]$. By maximum principle, we have at $(x_0,t_0)$
\begin{equation}
\label{E:proof-1}
\nabla f=0, \quad \Delta f \geq 0 \quad \text{ and } \quad \partial_t f \leq 0.
\end{equation}
We will show that this gives rise to a contradiction, provided that $\tilde{D}$ is sufficiently large depending only on $S$.

\vspace{1ex}

\textit{Claim: $H \geq 11\tilde{D}/6$ for all $t \in [0,T)$.}

\vspace{1ex}

\textit{Proof of Claim}: By (\ref{E:A-Atilde-bound}), we have $h_{ij} \geq \frac{11}{12} \tilde{D} g_{ij}$ at $t=0$ provided that $\tilde{D}$ is sufficiently large depending on $S$. Therefore, $H \geq \frac{11}{6} \tilde{D}$ at $t=0$ and the claim follows from Corollary \ref{C:H-preserve}. Note that Cauchy-Schwarz inequality implies that for all $t \in [0,T)$, we have whenever $\tilde{D} \geq 1$,
\[ |A|^2 \geq \frac{1}{2}H^2 \geq \frac{121}{72} \tilde{D}^2 \geq 1. \]
On the other hand, by Lemma \ref{L:MCF} (i) and (\ref{E:htilde-evolve}), we have at $(x_0,t_0)$ that  
\[ (\partial_t -\Delta) f \geq  |A|^2 \tilde{h}(v,v) +O(|A|^2) >0 \]
provided that $\tilde{h}(v,v)=\tilde{D}/2$ is sufficiently large depending on $S$. This contradicts (\ref{E:proof-1}). 

Finally, we show that $x_0$ cannot be a boundary point of $\Sigma$ either. Suppose $x_0 \in \partial \Sigma$. Since $\tilde{A}$ decomposes at the boundary by Lemma \ref{L:A-tilde-1}. We must have either $v=\partial_1$ or $v=\partial_2$ where ${\partial_1,\partial_2}$ is the orthonormal frame (with respect to $\Sigma_{t_0}$) from the Fermi coordinates at $x_0 \in \partial \Sigma$. Extend $v$ to a spacetime neighborhood of $(x_0,t_0)$ (note that $\partial \Sigma_{t_0}$ is convex so any point close to $x_0$ can be connected to $x_0$ by a radial geodesic) and define $f$ as before. To arrive at a contradiction, it suffices to show $\partial_1 f \geq 0$ at $(x_0,t_0)$ when $\tilde{D}>0$ is sufficiently large, depending only on $S$. If $\partial_1 f>0$, then $x_0$ cannot be a spatial minimum. If $\partial_1 f =0$, then the maximum principle can be applied to give the same contradiction as in the interior case. 

Suppose $v=\partial_1$ at $(x_0,t_0)$. Then $\tilde{h}_{22} \geq \tilde{h}_{11}=\tilde{D}/2$ at $(x_0,t_0)$. By Lemma \ref{L:A-tilde-1}, (\ref{E:N-h11-tilde}) and (\ref{A:S-bound}), we have at $(x_0,t_0)$ that
\begin{align*}
\partial_1 f= \nabla_1 \tilde{h}_{11}=& 2h^S_{22}H +(h^S_{\nu \nu} -3 h^S_{22}) h_{11}+ \nabla_\nu^S h^S_{22} \\
 \geq &  (h^S_{\nu \nu} + h^S_{22}) h_{11} +\nabla_\nu^S h^S_{22} \\
= & \frac{1}{2} H^S \tilde{D} +\nabla_\nu^S h^S_{22} \geq 0
\end{align*}
provided that $\tilde{D} \geq 2L_1$. 

Suppose now $v=\partial_2$ at $(x_0,t_0)$. Then $\tilde{h}_{11} \geq  \tilde{h}_{22}=\tilde{D}/2$ at $(x_0,t_0)$. By the claim above, Lemma \ref{L:A-tilde-1}, (\ref{E:N-h22-tilde}) and (\ref{A:S-bound}), we have at $(x_0,t_0)$ that
\begin{align*}
\partial_1 f = \nabla_1 \tilde{h}_{22}=& h^S_{22} H + (h^S_{\nu \nu} -3 h^S_{22}) h_{22} - \nabla_\nu^S h^S_{22}  \\
 \geq &  \frac{11}{6} h^S_{22} \tilde{D} + \frac{1}{2} (h^S_{\nu \nu} -3 h^S_{22}) \tilde{D} -\nabla_\nu^S h^S_{22} \\
 \geq & \frac{1}{3} H^S \tilde{D} -\nabla_\nu^S h^S_{22} \geq 0
\end{align*}
provided that $\tilde{D} \geq 3L_1$. This finishes the proof of Theorem \ref{T:tilde-convexity-I}.
\end{proof}

Note that dim $\Sigma=2$ is crucially used in the proof above so that one can extract a term involving $H$, on which we have a good lower bound. Using (\ref{E:A-Atilde-bound}), we immediately have the following corollary.

\begin{corollary}
\label{C:convexity}
There exists a constant $D_0=D_0(S)>0$ such that whenever $\{\Sigma_t\}_{t\in[0,T)}$ is a solution to the free-boundary mean curvature flow with 
\[ h_{ij} \geq D g_{ij} \quad \text{ at $t=0$}  \]
for some constant $D \geq D_0$, then we have
\[  h_{ij} > \frac{1}{3} D g_{ij}  \quad \text{ for all $t\in [0,T)$} .\]
\end{corollary}

\begin{remark}
It is easy to see that one can indeed choose $D=D_0=0$ in case $S$ is totally umbilic. This recovers the two-dimensional case of \cite[Theorem 4.4]{Stahl96b} which says that convexity is preserved throughout the flow for umbilic barrier surface $S$. In the non-umbilic case, we have shown that a convexity lower bound may not be preserved (c.f. \cite[Proposition 4.5]{Stahl96b}) but will at most decrease by a factor of $1/3$.
\end{remark}

From now on, we will assume that the hypothesis in Corollary \ref{C:convexity} is satisfied so the surfaces $\Sigma_t$ are convex for all $t \in [0,T)$. In particular, we always have $|A| \leq H$.

\section{Preservation of curvature pinching}
\label{S:pinching}

In this section, we want to derive another convexity pinching estimate, which is required to show that the rescaled flow converges to a shrinking half-sphere. 

As already observed in \cite{Stahl96b}, it is impossible to achieve the optimal estimate
\[ \epsilon  H g_{ij} \leq h_{ij} \leq \kappa H g_{ij} \]
for $0 < \epsilon \leq 1/2 < \kappa <1$ as in \cite{Huisken84}. A counterexample is given by $\Sigma_0$ which is a spherical cap intersecting the unit sphere $S=\mathbb{S}^2$ orthogonally but $\Sigma_t$ will not remain spherical for any $t>0$ (this example also shows that the flow is not $C^3$ up to $t=0$). However, we will establish a weaker pinching estimate in Corollary \ref{C:pinching} which is sufficient for our purpose.

We first generalize \cite[Theorem 4.8]{Stahl96b} to \emph{arbitrary} convex barrier surfaces for the perturbed second fundamental form.

\begin{theorem}
\label{T:pinching-pre}
There exists a constant $\tilde{D}_1=\tilde{D}_1(S)>0$ such that whenever $\{\Sigma_t\}_{t\in[0,T)}$ is a solution to the free-boundary mean curvature flow with 
\[ \tilde{h}_{ij} \geq \epsilon |\tilde{A}| g_{ij} +\tilde{D}g_{ij} \quad \text{ at $t=0$}  \]
for some constants $\tilde{D} \geq \tilde{D}_1$ and $\epsilon \in (0,1/100)$, then we have 
\[  \tilde{h}_{ij} > \frac{1}{2} ( \epsilon |\tilde{A}| g_{ij} +\tilde{D} g_{ij}) \quad \text{ for all $t\in [0,T)$} .\]
\end{theorem}

\begin{proof}
We argue by contradiction as in Theorem \ref{T:tilde-convexity-I}. Suppose there is a first time $t_0 \in (0,T)$ and a point $x_0 \in \Sigma$ such that 
\[ \tilde{h}(v,v)=\frac{1}{2}( \epsilon |\tilde{A}|+\tilde{D} )\]
for some unit tangent vector $v \in T_{x_0}\Sigma_{t_0}$. As before, we extend the vector $v$ locally satisfying (\ref{E:v-extend}) and consider the function
\[ f:=\tilde{h}(v,v) -  \frac{1}{2} ( \epsilon |\tilde{A}| +\tilde{D}  )g(v,v).\]
By the claim in the proof of Theorem \ref{T:tilde-convexity-I}, we have $H \geq 11\tilde{D}/6$ for all $t \in [0,T)$ and $1=O(|A|^2)$. Moreover, we can assume $|\tilde{A}| \geq 1$ by (\ref{E:A-Atilde-bound}).

If $x_0$ is an interior point of $\Sigma$, by Lemma \ref{L:MCF} (i), (\ref{E:htilde-evolve}) and Corollary \ref{C:|A-tilde|-evolve}, we have at $(x_0,t_0)$ that 
\[
(\partial_t -\Delta) f \geq  |A|^2 \left( \tilde{h}(v,v) - \frac{1}{2} \epsilon |\tilde{A}| \right) + O(|A|^2) >0
\]
provided that $\tilde{D}$ is sufficiently large. Hence $x_0$ cannot be an interior point of $\Sigma$.

Suppose now $x_0 \in \partial \Sigma$ and $v=\partial_1$. Then $\tilde{h}_{22} \geq \tilde{h}_{11}=\frac{1}{2} ( \epsilon |\tilde{A}|+\tilde{D} )$ at $(x_0,t_0)$. By Lemma \ref{L:A-tilde-1}, (\ref{E:N-h11-tilde}), (\ref{A:S-bound}), (\ref{E:N-|A-tilde|-ineq}) and Cauchy-Schwarz inequality, we have at $(x_0,t_0)$ that (recall that $\epsilon < \frac{1}{100}$)
\begin{align*}
\partial_1 f =& \nabla_1 \tilde{h}_{11} - \frac{1}{2} \epsilon N (|\tilde{A}|) \\
\geq &  2h^S_{22}H +(h^S_{\nu \nu} -3 h^S_{22}) h_{11}+ \nabla_\nu^S h^S_{22}  \\
& - \frac{\epsilon}{2} \Big(3 h^S_{22} \frac{H}{|A|} h_{11} + (h^S_{\nu\nu}-2h^S_{22})|\tilde{A}| + \sqrt{2} |\nabla_\nu^S h^S_{22}|\Big)\\
\geq & \left(h^S_{\nu \nu} +\left( 1-\frac{3 \epsilon}{\sqrt{2}} \right)h^S_{22} \right) h_{11}  - \frac{\epsilon}{2} |\tilde{A}|   (h^S_{\nu\nu}-2h^S_{22})  -2  |\nabla_\nu^S h^S_{22}|\\
\geq & \left(h^S_{\nu \nu} +\left( 1-\frac{3 \epsilon}{\sqrt{2}} \right)h^S_{22} \right) \frac{\tilde{D}}{2} -2  |\nabla_\nu^S h^S_{22}| \\
\geq & \frac{1}{4} H^S \tilde{D} -2  |\nabla_\nu^S h^S_{22}| \geq 0
\end{align*}
provided that $\tilde{D} \geq 8 L_1$. 

Suppose $x_0 \in \partial \Sigma$ and $v=\partial_2$. Then $\tilde{h}_{11} \geq \tilde{h}_{22}=\frac{1}{2} ( \epsilon |\tilde{A}|+\tilde{D} )$ at $(x_0,t_0)$. By Lemma \ref{L:A-tilde-1}, (\ref{E:N-h22-tilde}), (\ref{A:S-bound}), (\ref{E:N-|A-tilde|-ineq}) and Cauchy-Schwarz inequality, we have at $(x_0,t_0)$ that (recall that $|A| \leq H$ by convexity and $\epsilon < \frac{1}{100}$) 
\begin{align*}
\partial_1 f =&\nabla_1 \tilde{h}_{22} - \frac{1}{2} \epsilon N (|\tilde{A}|)  \\
\geq &  h^S_{22} H + (h^S_{\nu \nu} -3 h^S_{22}) h_{22} - \nabla_\nu^S h^S_{22} \\
& - \frac{\epsilon}{2} \Big(3 h^S_{22} \frac{H}{|A|} h_{11} + (h^S_{\nu\nu}-2h^S_{22})|\tilde{A}| + \sqrt{2} |\nabla_\nu^S h^S_{22}|\Big)\\
 \geq &  h^S_{22} \left(H- \frac{3\sqrt{2}}{2} \epsilon h_{11}\right) -\frac{\epsilon}{2} |\tilde{A}| h^S_{22}   + (h^S_{\nu \nu} -3 h^S_{22}) \frac{\tilde{D}}{2} -2  |\nabla_\nu^S h^S_{22}| \\
 \geq  & h^S_{22} \left( 1-\frac{3\sqrt{2}}{2}\epsilon-\frac{1}{2} \epsilon\right)H + (h^S_{\nu \nu} -3 h^S_{22}) \frac{\tilde{D}}{2} -2  |\nabla_\nu^S h^S_{22}| \\
 \geq & h^S_{22} (1-3 \epsilon) \frac{11\tilde{D}}{6}+  (h^S_{\nu \nu} -3 h^S_{22}) \frac{\tilde{D}}{2} -2  |\nabla_\nu^S h^S_{22}|\\
 \geq & \frac{1}{200} H^S \tilde{D} -2 |\nabla_\nu^S h^S_{22}| \geq 0
\end{align*}
provided that $\tilde{D} \geq 400 L_1$. This contradicts the maximum principle.
\end{proof}

We see again that it is important to have dim $\Sigma=2$ so that the positive term involving $H$ arises. Using (\ref{E:A-Atilde-bound}) and the Cauchy-Schwarz inequality $\frac{H}{\sqrt{2}} \leq |A|$, we immediately have the following corollary.

\begin{corollary}
\label{C:pinching}
There exists a constant $D_1=D_1(S)>0$ such that whenever $\{\Sigma_t\}_{t\in[0,T)}$ is a solution to the free-boundary mean curvature flow with 
\[ h_{ij} \geq D g_{ij} \quad \text{ at $t=0$}  \]
for some constant $D \geq D_1$ and
\[ h_{ij} \geq \epsilon |A| \quad \text{ at $t=0$},    \]
for some $\epsilon \in (0,1/100)$, then we have
\[  h_{ij} > \frac{\epsilon}{2 \sqrt{2}} H g_{ij} \quad \text{ for all $t\in [0,T)$} .\]
\end{corollary}

\section{Pinching estimate for the traceless second fundamental form}
\label{S:pinching-A-circle}

In this section, we use the Stampacchia iteration scheme to prove a pinching estimate for the traceless second fundamental form. This is the key ingredient to show that $\Sigma_t$ evolves to half of a ``round'' point. As in the previous sections, we need to first work with the perturbed second fundamental form $\tilde{A}$. The corresponding estimates for $A$ then follow. 

According to Corollary \ref{C:convexity} and \ref{C:pinching}, assuming that $\Sigma_0$ is sufficiently convex, then $\Sigma_t$ remains convex for all time and there exists a constant $\epsilon=\epsilon(\Sigma_0)>0$ such that
\begin{equation}
\label{E:lower-pinch}
h_{ij} \geq \epsilon H g_{ij} \quad \text{ for all $t \in [0,T)$}.
\end{equation}
Similarly, by Theorem \ref{T:pinching-pre}, we can also assume that there exists a constant $\tilde{\epsilon}=\tilde{\epsilon}(\Sigma_0,S)>0$ such that
\begin{equation}
\label{E:lower-pinch-tilde}
\tilde{h}_{ij} \geq \tilde{\epsilon} \tilde{H} g_{ij} \quad \text{ for all $t \in [0,T)$}.
\end{equation}
We shall always assume the two inequalities above in the rest of the paper.

First, we recall the following general result in \cite[Theorem 3.1]{Edelen16}. Note that we allow an extra term $|\Sigma_t|$ and $\int_{\Sigma_t} f^p$ (with coefficient depending possibly on $\beta$) in (\ref{E:POINCARE}) in contrast to the ``Poincar\'{e}-like'' inequality in \cite{Edelen16}. It is easy to see that the arguments still go through since this additional term can be absorbed into the corresponding terms in the ``Evolution-like'' inequality (\ref{E:EVOLUTION}). Moreover, the constants depending on $S$ in \cite{Edelen16} actually only depend on the constants $K,L_1,L_2$ in (\ref{A:Z-bound}) and (\ref{A:S-bound}).

\begin{theorem}
\label{T:Stampacchia}
Let $\{\Sigma_t\}_{t\in[0,T)}$ is a solution to the free-boundary mean curvature flow with $T < \infty$. Let $f_\alpha \geq 0$ be some function on $\Sigma_t$, depending on some parameters $\alpha=\alpha(S,\Sigma_0,T)$. Let $\tilde{G} \geq 0$ and $\tilde{H}>0$ be functions on $\Sigma_t$ such that
\begin{equation}
\label{A:Stamp-bound}
H=O(\tilde{H}), \qquad \nabla \tilde{H} = O(\tilde{G}).
\end{equation}
Let $f=f_\alpha \tilde{H}^\sigma$, and $f_k=(f-k)_+$, where $\sigma>0$ will be small and $k>0$ large. Write $A(k)=\{f \geq k\}$ and $A(k,t)=A(k) \cap \Sigma_t$.

Suppose $f$ satisfies the following inequalities: there exist positive constants $c=c(S,\Sigma_0,T,\alpha)$ and $C=C(S,\Sigma_0,T,\alpha, p,\sigma, \beta)$, such that for any $p > p_0(\alpha,c)$, $0<\sigma<1/2$, $k>0$ and $\beta>0$,
\begin{align}
\label{E:POINCARE}
\frac{1}{c} \int_{\Sigma_t} f^p \tilde{H}^2 \leq & p(1+\beta^{-1}) \int_{\Sigma_t} f^{p-2} |\nabla f|^2 +(1+\beta p) \int_{\Sigma_t} \frac{\tilde{G}^2}{\tilde{H}^{2-\sigma}} f^{p-1} \\
& + \int_{\partial \Sigma_t} f^{p-1} \tilde{H}^\sigma +C\left( \int_{\Sigma_t} f^p  +|\Sigma_t|\right), \nonumber
\end{align}
\begin{align}
\label{E:EVOLUTION}
\partial_t \int_{\Sigma_t} f^p_k \leq & -\frac{1}{3} p^2 \int_{\Sigma_t} f_k^{p-2} |\nabla f|^2 -\frac{p}{c} \int_{\Sigma_t} \frac{\tilde{G}^2}{\tilde{H}^{2-\sigma}} f^{p-1}_k +cp\sigma \int_{A(k,t)} \tilde{H}^2 f^p \\
& -\frac{1}{5} \int_{\Sigma_t} \tilde{H}^2 f^p_k +C \int_{A(k,t)} f^p +C|A(k)| +cp \int_{\partial \Sigma_t} f^{p-1}_k \tilde{H}^\sigma. \nonumber
\end{align}
Then, for $p$ sufficiently large, and $\sigma$ sufficiently small (depending on $p$), $f$ is uniformly bounded on $\Sigma \times [0,T)$ with the bound depending only on $S,\Sigma_0,T,\alpha,p$ and $\sigma$.
\end{theorem}

The main result of this section is the following:

\begin{theorem}
\label{T:pinching-A-tilde}
Under the assumption of (\ref{E:lower-pinch}) and (\ref{E:lower-pinch-tilde}), there exist constants $\tilde{C}_0 < \infty$ and $\sigma>0$, both depending only on $\Sigma_0$ and $S$ such that for all $t \in [0,T)$, we have the estimate
\begin{equation}
\label{E:pinching-A-tilde}
|\tilde{A}|^2 -\frac{1}{2} \tilde{H}^2 \leq \tilde{C}_0 \tilde{H}^{2-\sigma}.
\end{equation}
\end{theorem}

Our idea is to apply Theorem \ref{T:Stampacchia} to show that the non-negative function
\[
f:=\frac{|\ti A|^2-\frac12\ti H^2}{\ti  H^{2-\sigma}}
\]
is uniformly bounded in $\Sigma \times [0,T)$ for some suitable choice of the parameter $\sigma>0$. We first observe that from Lemma \ref{L:A-tilde-1}, \ref{L:N-htilde} and (\ref{E:N-|A-tilde|-ineq}) that along $\partial \Sigma$, we have
\begin{equation}
\label{E:N-bounds-1}
N\tilde{H} = h^S_{\nu\nu} \tilde{H},\qquad |N|\tilde{A}|| \leq C(K,L_1) \tilde{H},
\end{equation}
where we have also used $|\tilde{A}| \leq \tilde{H}$ since $\tilde{A}>0$ for all time. From (\ref{E:N-bounds-1}) we obtain
\begin{equation}
\label{E:f-bound-1}
Nf =O(\tilde{H}^\sigma).
\end{equation}
Moreover, it follows from the definition of $f$ that on $\Sigma \times [0,T)$, we have
\begin{equation}
\label{E:f-bound-2}
0 \leq f \leq \tilde{H}^\sigma.
\end{equation}

We first show that $f$ satisfies the ``Poincar\'{e}-like'' inequality (\ref{E:POINCARE}) with $\tilde{G} = |\nabla \tilde{H}|$ such that (\ref{A:Stamp-bound}) is clearly satisfied.

\begin{lemma}
\label{L:POINCARE}
There exists a constants $c=c(S,\Sigma_0)>0$ and $C=C(S,\Sigma_0,\sigma,p,\beta)>0$ such that for any $\beta>0$, $0<\sigma<1/2$ and $p > 4$, we have for all $t \in [0,T)$,
\begin{align*}
\frac{1}{c}\int_{\Sigma_t} f^p\ti H^2 \le &p(1+\beta^{-1})\int_{\Sigma_t} |\nabla f|^2f^{p-2} + (1+p \beta)\int_{\Sigma_t} \frac{|\nabla \ti  H|^2}{\ti H^{2-\sigma}}f^{p-1}\\
&+\int_{\partial \Sigma_t}f^{p-1}\ti H^\sigma +C\left( \int_{\Sigma_t} f^p +|\Sigma_t|\right).
\end{align*}
\end{lemma}

\begin{proof}
We start by observing 
\[ |\ti A|^2=|A|^2+|P^\Sigma|^2+2\langle A, P^\Sigma \rangle,\]
\[ \ti H^2=H^2+V^2+2HV \]
where $V:=\Tr_g P^\Sigma$.
By a direct computation exactly as in \cite[Lemma 5.2 and 5.4]{Huisken84}, we obtain
\begin{equation}
\label{E:nabla-f}
\nabla f = \frac{1}{\tilde{H}^{2-\sigma}} \nabla|\tilde{A}|^2 + \left( \frac{\sigma}{\tilde{H}} f - \frac{2 |\tilde{A}|^2}{\tilde{H}^{3-\sigma}} \right) \nabla \tilde{H}
\end{equation}

\begin{align}
\label{E:Delta-f}
\Delta f=&\frac{\ti H\Delta |\ti A|^2-(2-\sigma)|\ti A|^2\Delta\ti H}{\ti H^{3-\sigma}}-\frac\sigma {2\ti H^{1-\sigma}}\Delta \ti H-\frac{2(2-\sigma)}{\ti H^{3-\sigma}} \langle \nabla |\ti A|^2, \nabla \ti H \rangle\\
&+(2-\sigma)(3-\sigma)\frac{|\ti A|^2}{\ti H^{4-\sigma}}|\nabla \ti H|^2 +\frac{\sigma(1-\sigma)}{2\ti H^{2-\sigma}}|\nabla \ti H|^2. \nonumber \\
=& \frac{\Delta |\ti A|^2 -\tilde{H} \Delta \tilde{H} -2|\nabla \tilde{A}|^2}{\ti H^{2-\sigma}} -\frac{2-\sigma}{\ti H} f \Delta \tilde{H} +\frac{2}{\tilde{H}^{4-\sigma}} |\tilde{H} \nabla\tilde{A} - \tilde{A} \nabla \tilde{H}|^2   \nonumber \\
&   +\frac{\sigma(1-\sigma)}{\tilde{H}^2} f |\nabla \tilde{H}|^2-\frac{2(1-\sigma)}{\tilde{H}} \langle \nabla \tilde{H},\nabla f\rangle.   \nonumber
\end{align}
Since $\Sigma_t$ is convex for all time, we have $|A| \leq H$. From the proof of Proposition \ref{P:P-estimate} and (\ref{E:P-C^2}), we obtain
\begin{equation}
\label{E:Delta-H-tilde}
\Delta \ti H=\Delta H+O(1+|\nabla H|+H^2).
\end{equation}
To compute the term $\Delta |\tilde{A}|^2$, we apply the standard Simons' identity to obtain a Simons'-type identity for the perturbed second fundamental form:
\begin{align*}
\frac{1}{2} \Delta |\ti A|^2=&\frac{1}{2}\Delta |A|^2+\frac{1}{2}\Delta |P^\Sigma|^2+\Delta\langle A, P^\Sigma \rangle \\
=&(\langle h_{ij},\nabla_i\nabla_jH \rangle+|\nabla A|^2+Z) + \langle P^\Sigma,\Delta P^\Sigma \rangle + |\nabla P^\Sigma|^2\\
& \langle \nabla_i\nabla_jH+Hh_{i\ell}h^{\ell}_{\phantom{\ell}j}-|A|^2h_{ij}, P^\Sigma_{ij} \rangle +\langle A,\Delta P^\Sigma \rangle +2 \langle \nabla A,\nabla P^\Sigma \rangle\\
=&\langle \tilde{h}_{ij},\nabla_i\nabla_jH \rangle+|\nabla \tilde{A}|^2+Z + \langle P^\Sigma,\Delta P^\Sigma \rangle+\langle A,\Delta P^\Sigma \rangle\\
& +\langle Hh_{i\ell}h^{\ell}_{\phantom{\ell}j}-|A|^2h_{ij}, P^\Sigma_{ij} \rangle \\
=&\langle \tilde{h}_{ij},\nabla_i\nabla_jH \rangle+|\nabla \ti A|^2+ Z+O(1+H^3+|\nabla H|+H|\nabla H|).
\end{align*}
where $Z:=Hh_{ik}h^k_{\phantom{k}\ell}h^{\ell i}-|A|^4$ as in \cite[Section 2]{Huisken84} and we have used (\ref{E:P-C^2}), the proof of Proposition \ref{P:P-estimate} and $|A| \leq H$ in the last equality.
Observe that $H=\tilde{H}+O(1)$ by (\ref{E:A-Atilde-bound}) and $\nabla H=\nabla \tilde{H}+O(1+H)$ by (\ref{E:nablaA-Atilde-bound}). Moreover, we can assume $\tilde{H} \geq 1$ since $\Sigma_0$ is sufficiently convex and any lower bound of $H$ is preserved throughout the flow. Therefore, we obtain 
\begin{equation}
\label{E:Simons}
\frac{1}{2} \Delta |\ti A|^2=\ti h^{ij}\nabla_i\nabla_jH+|\nabla \ti A|^2+ Z+O(\tilde{H}^3+ \tilde{H}|\nabla \tilde{H}|).
\end{equation}
Putting (\ref{E:Delta-H-tilde}) and (\ref{E:Simons}) back into the first term in the last equation of (\ref{E:Delta-f}) and proceeding as in the proof of \cite[Lemma 5.4]{Huisken84}, we have
\begin{align*}
\Delta f=&\frac{2}{\ti H^{2-\sigma}}\langle \ti h^0_{ij},\nabla_i\nabla_jH \rangle +\frac{2Z}{\ti H^{2-\sigma}}+\frac{2}{\tilde{H}^{4-\sigma}} |\tilde{H} \nabla\tilde{A} - \tilde{A} \nabla \tilde{H}|^2 \\
&-\frac{2-\sigma}{\ti H}f\Delta \ti H +\frac{\sigma(1-\sigma)}{\tilde{H}^2} f |\nabla \tilde{H}|^2-\frac{2(1-\sigma)}{\tilde{H}} \langle \nabla \tilde{H},\nabla f\rangle \\
&+O\left(\ti H^{1+\sigma}+\frac1{\ti H^{1-\sigma}}|\nabla \ti H|\right),
\end{align*}
where $\tilde h^0_{ij}$ denotes the trace free part of $\tilde A$. Dropping two non-negative terms, we have
\begin{align*}
\Delta f\ge&\frac{2}{\ti H^{2-\sigma}}\langle \ti h^0_{ij},\nabla_i\nabla_jH \rangle +\frac{2Z}{\ti H^{2-\sigma}} -\frac{2-\sigma}{\ti H}f\Delta \ti H -\frac{2(1-\sigma)}{\tilde{H}} \langle \nabla \tilde{H},\nabla f\rangle \\
& +O\left(\ti H^{1+\sigma}+\frac1{\ti H^{1-\sigma}}|\nabla \ti H|\right).
\end{align*}
Since $\Sigma_t$ remains convex and (\ref{E:lower-pinch}) holds, we can apply the estimate in \cite[Lemma 2.3]{Huisken84} and use the bounds (\ref{E:A-Atilde-bound}) to give the inequality
\[ Z\ge 2\epsilon^2 H^2 \left(|A|^2-\frac{1}{2} H^2 \right) = 2\epsilon^2 \tilde{H}^2 \left(|\tilde{A}|^2-\frac{1}{2} \tilde{H}^2 \right) +O(\tilde{H}^3).\]
Therefore, we obtain the following differential inequality
\begin{align}
\label{E:Delta-f-ineq}
\Delta f\ge&\frac{2}{\ti H^{2-\sigma}}\langle \ti h^0_{ij},\nabla_i\nabla_jH \rangle +4 \epsilon^2 f \tilde{H}^2 -\frac{2-\sigma}{\ti H}f\Delta \ti H -\frac{2(1-\sigma)}{\tilde{H}} \langle \nabla \tilde{H},\nabla f\rangle \\
&-C(S)\left(\ti H^{1+\sigma}+\frac1{\ti H^{1-\sigma}}|\nabla \ti H|\right). \nonumber
\end{align}
We will multiply the inequality by $f^{p-1}$ and integrate by parts as in \cite[P.248]{Huisken84}. Since there are new boundary terms showing up and errors terms to be absorbed, let us look at the terms that are integrated by part more carefully. Since we have the bounds (\ref{E:N-bounds-1}), (\ref{E:f-bound-1}), (\ref{E:f-bound-2}), together with Peter-Paul inequality, we have for any $\beta>0$,
\begin{align*}
\int_\Sigma f^{p-1} \Delta f =& -(p-1) \int_\Sigma f^{p-2} |\nabla f|^2 +\int_{\partial \Sigma} f^{p-1} N(f) \\
\leq & -(p-1) \int_\Sigma f^{p-2} |\nabla f|^2 + C(S)\int_{\partial \Sigma} f^{p-1} \tilde{H}^\sigma,
\end{align*}
\begin{align*} 
\int_\Sigma \frac{1}{\tilde{H}} & f^p \Delta \tilde{H} =\int_\Sigma \frac{1}{\tilde{H}^2} f^p |\nabla \tilde{H}|^2 -p \int_\Sigma \frac{1}{\tilde{H}} f^{p-1} \langle \nabla \tilde{H},\nabla f \rangle +\int_{\partial \Sigma} \frac{1}{\tilde{H}} f^p N(\tilde{H}) \\
\leq &(1+p \beta ) \int_\Sigma \frac{1}{\tilde{H}^{2-\sigma}} f^{p-1} |\nabla \tilde{H}|^2 + (4\beta)^{-1}p \int_\Sigma f^{p-2} |\nabla f|^2 +C(S)\int_{\partial \Sigma} f^{p-1} \tilde{H}^\sigma,
\end{align*}
By Cauchy-Schwarz, we have
\[
\int_\Sigma \frac{1}{\tilde{H}} f^{p-1} \langle \nabla \tilde{H},\nabla f \rangle \leq  \frac{1}{2} \int_\Sigma \frac{1}{\tilde{H}^{2-\sigma}} f^{p-1} |\nabla \tilde{H}|^2 + \frac{1}{2} \int_\Sigma f^{p-2} |\nabla f|^2.
\]
The term involving the $\tilde{h}^0_{ij}$ requires more work. Recall that $\|\tilde{h}^0_{ij}\|^2=f \tilde{H}^{2-\sigma} \leq \tilde{H}^2$ and from Codazzi equation $\nabla_i h^0_{ij} = \frac{1}{2} \nabla_j H$. We will need the crucial fact from Lemma \ref{L:A-tilde-1} that the (trace-free) perturbed second fundamental form decomposes along $\partial \Sigma$ to estimate the boundary term. Furthermore, using the uniform bound (\ref{E:nablaA-Atilde-bound}) and Peter-Paul inequality, we obtain for any $\beta>0$
\begin{align*}
&-2 \int_\Sigma \frac{1}{\tilde{H}^{2-\sigma}} f^{p-1} \langle \ti h^0_{ij},\nabla_i\nabla_jH \rangle \\
= & -2(2-\sigma) \int_\Sigma \frac{1}{\tilde{H}^{3-\sigma}} f^{p-1} \langle \tilde{h}^0_{ij}, \nabla_i \tilde{H} \nabla_j H\rangle +2 (p-1) \int_\Sigma \frac{1}{\tilde{H}^{2-\sigma}} f^{p-2} \langle \tilde{h}^0_{ij}, \nabla_i f \nabla_j H\rangle \\
& + \int_\Sigma \frac{1}{\tilde{H}^{2-\sigma}} f^{p-1} |\nabla H|^2 +\int_\Sigma f^{p-1} O\left(\frac{ |\nabla H|}{\tilde{H}^{1-\sigma}} \right)  -2\int_{\partial \Sigma} \frac{1}{\tilde{H}^{2-\sigma}} f^{p-1} \tilde{h}^0_{11} N(H)\\
\leq & 4 \int_\Sigma  \frac{1}{\tilde{H}^{2-\frac{\sigma}{2}}} f^{p-\frac{1}{2}} |\nabla \tilde{H}| |\nabla H| +2p \int_\Sigma \frac{1}{\tilde{H}^{1-\frac{\sigma}{2}}} f^{p-\frac{3}{2}} |\nabla f| |\nabla H|   \\
& + \int_\Sigma \frac{1}{\tilde{H}^{2-\sigma}} f^{p-1} |\nabla H|^2  +\int_\Sigma f^{p-1} O\left(\frac{|\nabla H| }{\tilde{H}^{1-\sigma}} \right) +C(S)\int_{\partial \Sigma} f^{p-1} \tilde{H}^\sigma\\
\leq & 4 \int_\Sigma  \frac{1}{\tilde{H}^{2-\frac{\sigma}{2}}} f^{p-\frac{1}{2}} |\nabla \tilde{H}|^2+2p \int_\Sigma \frac{1}{\tilde{H}^{1-\frac{\sigma}{2}}} f^{p-\frac{3}{2}} |\nabla f| |\nabla H|     \\
& + \int_\Sigma \frac{1}{\tilde{H}^{2-\sigma}} f^{p-1} |\nabla \tilde{H}|^2 +\int_\Sigma f^{p-1} O\left(\frac{|\nabla H|}{\tilde{H}^{1-\sigma}}  \right) +C(S)\int_{\partial \Sigma} f^{p-1} \tilde{H}^\sigma\\
\leq & \beta^{-1}p \int_\Sigma f^{p-2} |\nabla f|^2 + (\beta p+5) \int_\Sigma \frac{1}{\tilde{H}^{2-\sigma}} f^{p-1} |\nabla \tilde{H}|^2    \\
& +(\beta p+1)\int_\Sigma f^{p-1} O\left(\tilde{H}^\sigma+\frac{ |\nabla \tilde{H}|}{\tilde{H}^{1-\sigma}} \right) +C(S)\int_{\partial \Sigma} f^{p-1} \tilde{H}^\sigma
\end{align*}
where we have used the estimates $|\nabla H|=|\nabla \tilde{H}| +O(\tilde{H})$ and $|\nabla H|^2 =|\nabla \tilde{H}|^2 +O(\tilde{H} |\nabla \tilde{H}| + \tilde{H}^2)$. Putting all of these estimates back into (\ref{E:Delta-f-ineq}), since $p \geq 2$, we have for any $\beta>0$,
\begin{align}
\label{E:POINCARE-proof}
4\epsilon^2 \int_\Sigma f^p \tilde{H}^2 \leq & 2\beta^{-1}p \int_\Sigma f^{p-2} |\nabla f|^2 +(3\beta p+8) \int_\Sigma \frac{1}{\tilde{H}^{2-\sigma}} f^{p-1} |\nabla \tilde{H}|^2 \\
&+(\beta p+2)\int_\Sigma f^{p-1} O\left(\tilde{H}^{1+\sigma}+\frac{|\nabla \tilde{H}|}{\tilde{H}^{1-\sigma}}  \right) +C(S) \int_{\partial \Sigma} f^{p-1} \tilde{H}^\sigma. \nonumber
\end{align}
It remains to control the error term. Note that by \cite[Remark 3.2]{Edelen16}, for any arbitrary function $g \geq 0$ on $\Sigma=\Sigma_t$. If $r \in (0,2)$ and $q \in (0,p)$ with $rp/q<2$, then for any $\mu>0$, we have
\begin{equation}
\label{E:Young}
\int_\Sigma g^q \tilde{H}^r \leq \mu^{-1} \int_\Sigma g^p \tilde{H}^2 +C(\mu,r,p,q) \int_\Sigma g^p + |\Sigma_t|.
\end{equation}
Therefore, we have for any $\mu>0$, when $p> 2/(1-\sigma)$,
\[ \int_\Sigma f^{p-1} \tilde{H}^{1+\sigma} \leq \mu^{-1} \int_\Sigma f^p \tilde{H}^2+C(\mu,\sigma,p) \int_\Sigma f^p + |\Sigma_t| .\]
On the other hand, by Cauchy-Schwarz
\[ \int_\Sigma \frac{1}{\tilde{H}^{1-\sigma}} f^{p-1} |\nabla \tilde{H}| \leq \frac{1}{2} \int_\Sigma \frac{1}{\tilde{H}^{2-\sigma}} f^{p-1} |\nabla \tilde{H}|^2 + \frac{1}{2} \int_\Sigma f^{p-1} \tilde{H}^\sigma. \]
We can then estimate the error term to be
\begin{align*}
(\beta p+2)\int_\Sigma f^{p-1} & O\left(\tilde{H}^{1+\sigma}+\frac{1}{\tilde{H}^{1-\sigma}} |\nabla \tilde{H}| \right) \\
 \leq  & C(S,\beta,p) \mu^{-1} \int_\Sigma f^p \tilde{H}^2  +  C(S)(\beta p+2) \int_\Sigma \frac{1}{\tilde{H}^{2-\sigma}} f^{p-1} |\nabla \tilde{H}|^2 \\
&  +C(S,\mu,\sigma,p,\beta) \int_\Sigma f^p + C(S,p,\beta) |\Sigma_t|
\end{align*}
If we choose $\mu>0$ sufficiently large, depending only on $S$, $\beta$, $\epsilon$ and $p$, then we can absorb the first term to the left hand side of (\ref{E:POINCARE-proof}). This finishes the proof.
\end{proof}

Next, we want to establish an ``Evolution-like'' inequality (\ref{E:EVOLUTION}) for $f$. Before that, we prove a useful lemma in the same spirit as \cite[Lemma 2.3 (ii)]{Huisken84}. Note that we have a worse error term of order $H^4$ as opposed to the one of order $H^2$ in \cite[Lemma 2.3 (ii)]{Huisken84}. Nonetheless, this is still sufficient for our later purpose.

\begin{lemma}
\label{L:squeeze}
Under the assumption (\ref{E:lower-pinch-tilde}) and $\tilde{H}>0$ for all time, we have
\[
|\ti H \nabla \ti A-\ti A \nabla\ti H|^2\ge \frac14 \tilde{\epsilon}^2\ti H^2|\nabla \ti H|^2-\frac{C(S)}{4\tilde{\epsilon}^2} \ti H^4.
\]
\end{lemma}

\begin{proof}
Decomposing into symmetric and skew-symmetric parts with respect to $i$ and $k$, we have
\begin{align*}
 \nabla_i\ti h_{kl} \cdot \ti H -  \nabla_i\ti  H \cdot \ti h_{kl}=\frac{1}{2} \tilde{E}_{ikl}+ \frac{1}{2} \tilde{F}_{ikl},
\end{align*}
where 
\begin{align*}
\tilde{E}_{ikl}=& (\nabla_i \tilde{h}_{kl}+\nabla_k \tilde{h}_{il}) \cdot \ti H - (\nabla_i \tilde{H} \cdot \tilde{h}_{kl} +\nabla_k \tilde{H} \cdot \tilde{h}_{il}),\\
\tilde{F}_{ikl}=&(\nabla_i \tilde{h}_{kl}-\nabla_k \tilde{h}_{il}) \cdot \ti H-  (\nabla_i \tilde{H} \cdot \tilde{h}_{kl} -\nabla_k \tilde{H} \cdot \tilde{h}_{il}).
\end{align*}
Keeping only the skew-symmetric component $\tilde{F}_{ikl}$ and we have
\begin{align*}
|\nabla_i\ti h_{kl} \cdot \ti H -  \nabla_i\ti  H \cdot \ti h_{kl}|^2 \geq & \frac{1}{4}|(\nabla_i \tilde{h}_{kl}-\nabla_k \tilde{h}_{il}) \cdot \ti H-  (\nabla_i \tilde{H} \cdot \tilde{h}_{kl} -\nabla_k \tilde{H} \cdot \tilde{h}_{il})|^2 \\
\geq & \frac{1}{4} |\nabla_i \tilde{H} \cdot \tilde{h}_{kl} -\nabla_k \tilde{H} \cdot \tilde{h}_{il} |^2  - \frac{1}{2}\tilde{H}  |\nabla \tilde{H}| |\tilde{A}| |\nabla_i P^\Sigma_{kl}-\nabla_k P^\Sigma_{il}|
\end{align*}
Arguing as in \cite[Lemma 2.3(ii)]{Huisken84} and using (\ref{E:lower-pinch-tilde}), we have
\[  \frac{1}{4} |\nabla_i \tilde{H} \cdot \tilde{h}_{kl} -\nabla_k \tilde{H} \cdot \tilde{h}_{il} |^2  \geq  \frac{1}{2} \tilde{\epsilon}^2 \tilde{H}^2 |\nabla \tilde{H}|^2.\]
To estimate the second term, we have the following perturbed Codazzi equation from the proof of Proposition \ref{P:P-estimate},
\[ 
\nabla_i \tilde{h}_{kl}-\nabla_k \tilde{h}_{il} = \nabla_i P^\Sigma_{kl} -\nabla_k P^\Sigma_{il} = O(\tilde{H}).
\]
Hence, by $|\tilde{A}| \leq \tilde{H}$ and Peter-Paul inequality, we have
\[ \frac{1}{2} \tilde{H} |\nabla \tilde{H}| |\tilde{A}| |\nabla_i P^\Sigma_{kl}-\nabla_k P^\Sigma_{il}| \leq \frac{C(S)}{2} \tilde{H}^3 |\nabla \tilde{H}| \leq \frac{1}{4} \tilde{\epsilon}^2 \tilde{H}^2 |\nabla \tilde{H}|^2 + \frac{C_S}{4 \tilde{\epsilon}^2} \tilde{H}^4.\]
Plugging them back gives the desired inequality.
\end{proof}

Now, we are ready to prove an ``Evolution-like'' inequality (\ref{E:EVOLUTION}) for $f$.

\begin{lemma}
\label{L:EVOLUTION}
There exists a constant $C=C(S,\Sigma_0,p,\sigma)>0$ such that
\begin{align*}
\frac{d}{dt} \int_{\Sigma_t} f_k^p
\le& -\frac{1}{3} p^2 \int_{\Sigma_t}f_k^{p-2}  |\nabla f|^2 -\frac {p \tilde{\epsilon}^2}{4}\int_{\Sigma_t}\frac{|\nabla \ti H|^2}{\ti H^{2-\sigma}}f_k^{p-1}+2\sigma p\int_{A(k,t)} \ti H^2f^p\\
& -\int_\Sigma \tilde{H}^2 f^p_k +C(S)p \int_{\partial \Sigma_t}f_k^{p-1}\ti H^\sigma+C\left(\int_{A(k,t)} f^p + |A(k,t)|\right)
\end{align*}
where $f_k=(f-k)_+$ and $A(k,t)=\{f \geq k\} \cap \Sigma_t$.
\end{lemma}

\begin{proof}
We first have to derive a good evolution inequality for $f$. To this end, we compute as in \cite[Lemma 5.2]{Huisken84} using (\ref{E:Htilde-evolve}), (\ref{E:|Atilde|^2-evolve}) and (\ref{E:A-Atilde-bound}),
\begin{align*}
\partial_t f =  \frac{\tilde{H} \Delta |\tilde{A}|^2 - (2-\sigma) |\tilde{A}|^2 \Delta \tilde{H}}{\tilde{H}^{3-\sigma}} -\frac{\sigma}{2} \frac{\Delta \tilde{H}}{\tilde{H}^{1-\sigma}} - \frac{2}{\tilde{H}^{2-\sigma}} |\nabla \tilde{A}|^2 +\sigma |\tilde{A}|^2 f + O(\tilde{H}^{1+\sigma}).
\end{align*}
Combining this with (\ref{E:Delta-f}), and using Lemma \ref{L:squeeze}, we obtain the inequality
\begin{align*}
(\partial_t -\Delta) f \leq & -\frac{2}{\tilde{H}^{4-\sigma}} |\tilde{H} \nabla \tilde{A} -\tilde{A} \nabla \tilde{H}|^2 + \sigma |\tilde{A}|^2 f +\frac{2(1-\sigma)}{\tilde{H}} \langle \nabla \tilde{H}, \nabla f \rangle +O(\tilde{H}^{1+\sigma}) \\
\leq & \frac{2(1-\sigma)}{\tilde{H}} \langle \nabla \tilde{H}, \nabla f \rangle  - \frac{\tilde{\epsilon}^2}{2} \frac{1}{\tilde{H}^{2-\sigma}} |\nabla \tilde{H}|^2 + \sigma |\tilde{A}|^2 f +\tilde{\epsilon}^{-2} O\left(\tilde{H}^{1+\sigma} \right).
\end{align*}

Multiply the inequality above by $p f^{p-1}_k$ and then integrate by parts as in \cite[Lemma 5.5, 5.7]{Huisken84}, we have
\begin{align*}
\partial_t \int_\Sigma f_k^p & + \frac{p(p-1)}{2} \int_{A(k,t)} f_k^{p-1} |\nabla f|^2 + \frac{p}{4} \tilde{\epsilon}^2 \int_\Sigma \frac{1}{\tilde{H}^{2-\sigma}} f_k^{p-1} |\nabla \tilde{H}|^2 + \int_\Sigma H^2 f^p_k \\
\leq &  \sigma p \int_{A(k,t)} \tilde{H}^2 f^{p-1}_k f +\tilde{\epsilon}^{-2} p \int_{A(k,t)} f^{p-1}_k O\left(\tilde{H}^{1+\sigma}\right)    + p \int_{\partial \Sigma} f^{p-1}_k N(f)
\end{align*}
Note that by (\ref{E:A-Atilde-bound}) and (\ref{E:f-bound-1}), we have 
\begin{align*}
-\int_\Sigma H^2 f^p_k \leq & -\int_\Sigma \tilde{H}^2 f^p_k + \int_{A(k,t)} f^p_k O(\tilde{H}) \\
\leq &  -\int_\Sigma \tilde{H}^2 f^p_k + \int_{A(k,t)} f^{p-1}_k O\left(\tilde{H}^{1+\sigma}\right),
\end{align*}
and
\[ p \int_{\partial \Sigma} f^{p-1}_k N(f) \leq C(S) p \int_{\partial \Sigma} f^{p-1}_k \tilde{H}^\sigma. \]

Finally, it remains to estimate the error term. First of all, applying Young's inequality (\ref{E:Young}) we get
\begin{align*}
C(S)(1+\tilde{\epsilon}^{-2}  p)  \int_{A(k,t)} f^{p-1}_k \tilde{H}^{1+\sigma} \leq & \sigma p \int_{A(k,t)} f^p_k \tilde{H}^2 \\
& + C(S,\tilde{\epsilon},\sigma,p) \left(\int_{A(k,t)} f^p_k  + |A(k,t)| \right). 
\end{align*}
\end{proof}

We now apply all the results above to prove Theorem \ref{T:pinching-A-tilde}.

\begin{proof}[Proof of Theorem \ref{T:pinching-A-tilde}]
Lemma \ref{L:POINCARE} and \ref{L:EVOLUTION} imply that we can apply Theorem \ref{T:Stampacchia} to the function $f$ so that for some fixed large $p$ (depending on $S$ and $\Sigma_0$) and small $\sigma$ (depending on $p$), there exists a constant $\tilde{C}_0=\tilde{C}_0(S,\Sigma_0)>0$ such that on the entire spacetime $\Sigma \times [0,T)$, we have
\[ \frac{|\ti A|^2-\frac12\ti H^2}{\ti H^{2-\sigma}}\le \tilde{C}_0<\infty.\]
\end{proof}

Using (\ref{E:pinching-A-tilde}) together with the bound (\ref{E:A-Atilde-bound}) and Peter-Paul inequality, we have for any $\eta>0$
\begin{align*}
|A|^2-\frac{1}{2} H^2 \leq & |\tilde{A}|^2 -\frac{1}{2} \tilde{H}^2 +C(S) \tilde{H} \\
\leq & \tilde{C}_0 \tilde{H}^{2-\sigma} +C(S) \tilde{H} \\
\leq & \frac{\eta}{2} \tilde{H}^2 + C(S,\eta,\Sigma_0)  +C(S) \tilde{H}\\
\leq & \frac{\eta}{2} H^2 + C(S,\eta,\Sigma_0) +C(S) H\\
\leq & \eta H^2 +C(S,\eta,\Sigma_0).
\end{align*}
Therefore, we have the following corollary.

\begin{corollary}
\label{C:pinch-bound}
For any $\eta>0$, we have
\[ |A|^2 -\frac{1}{2} H^2 \leq \eta H^2 +C(S,\eta,\Sigma_0).\]
\end{corollary}

\section{Gradient estimate for the mean curvature}
\label{S:H-gradient-estimate}

In this section, we derive a gradient estimate for the mean curvature, which can be used to compare the mean curvature at different points. Together with all the previous parts, our main result Theorem \ref{T:main} then follows from standard arguments as in \cite{Huisken84}. Note that we only need the gradient estimate below with $\eta>0$ small.

\begin{theorem}
\label{T:H-gradient-bound}
Under the assumption of (\ref{E:lower-pinch}) and (\ref{E:lower-pinch-tilde}), there exists $\eta_0=\eta_0(S)>0$ such that for each $0 < \eta <\eta_0$, there exists a constant $C=C(S,\eta,\Sigma_0)$ such that 
\[ |\nabla H|^2 \leq \eta H^4 +C(S,\eta,\Sigma_0)\]
holds on $\Sigma \times [0,T)$.
\end{theorem}

Let $\eta>0$ be fixed. WLOG, we assume $\eta < \min\{ (4K)^{-1},1\}$. As in \cite{Huisken84} and \cite{Edelen16}, we consider the following test functions defined on $\Sigma \times [0,T)$ by
\[ g:=  \frac{|\nabla H-h^S_{\nu\nu}H\nu_S^T|^2}H+bH\Big(|\ti A|^2-\frac12 \tilde{H}^2\Big)+ba|\ti A|^2-\eta e^{\frac{1}{\eta} \rho} H^3 +c\]
where $a,b,c$ are positive constants to be determined later. Here, $h^S$ is the second fundamental form of $S$ extended to $\mathbb{R}^3$ as in Section \ref{SS:barrier}, and $\nu_S^T$ is the tangential component (with respect to $\Sigma_t$) of the extended unit normal $\nu_S$. Moreover, $\rho$ is a function depending on the parameter $\eta$ defined by (recall the signed distance function $d$ to the barrier $S$ and the radial cutoff function $\chi$ from Section \ref{SS:barrier})
\[
\rho (x):=d(x)\chi\left(\frac{|d(x)|}{\eta}\right).
\]
From this definition and a similar calculation as in Section \ref{SS:barrier}, we know that $\rho$ is supported in the tubular neighborhood $S_{2\eta}$ and satisfies the bounds (using the bounds in Section \ref{SS:barrier})
\begin{equation}
\label{E:rho-estimate}
\|\rho\|_{C_0(\R^3)}\le 2 \eta, \quad \|D\rho\|_{C_0(\R^3)}\le 5 \quad \text{and} \quad \|D^2\rho\|_{C_0(\R^3)}\le \frac{15}{\eta}.
\end{equation}

Restricting the function $\rho$ to the evolving surface $\Sigma=\Sigma_t$ and using the formula $\Delta \rho=\textrm{tr}_{\Sigma} D^2 \rho -H D_\nu \rho$, we have the estimates
\begin{equation}
\label{E:rho-estimate-2}
|(\partial_t - \Delta)\rho| \leq \frac{30}{\eta}.
\end{equation}
Furthermore, we have $N \rho \equiv 1$ along $\partial \Sigma$. Using these, if we let $\zeta:=\eta e^{\frac{1}{\eta} \rho}$, then along $\partial \Sigma$ we have
\begin{equation}
\label{E:zeta-boundary}
\zeta \equiv \eta \qquad \text{ and } \qquad N \zeta \equiv 1.
\end{equation}
Furthermore, from (\ref{E:rho-estimate}) and (\ref{E:rho-estimate-2}), we have on $\Sigma$ the following estimates
\begin{equation}
\label{E:zeta-interior}
\eta e^{-2} \le\zeta\le \eta e^{2}, \quad  |\nabla\zeta| \leq 5e^{2} \quad \text{and} \quad |(\partial_t - \Delta)\zeta| \leq \frac{55 e^{2}}{\eta}
\end{equation}
These properties of $\zeta$ will become crucial in the proof of Theorem \ref{T:H-gradient-bound}.

For the proof of Theorem \ref{T:H-gradient-bound} we begin with computing the boundary derivatives of the terms appearing in $g$.

\begin{lemma}
Along $\partial \Sigma$, we have for all $t>0$,
\begin{align}
\label{E:N-term-1}
N (|\nabla H-h^S_{\nu\nu}H\nu^T_S|^2)= & 2(h^S_{\nu \nu} -h^S_{22}) |\nabla H-h^S_{\nu\nu}HV|^2 \\
& + 2(\nabla^S_2 h^S_{\nu \nu} +2 h_{22} h^S_{2 \nu}) H (\partial_2 H), \nonumber
\end{align}
\begin{align}
\label{E:N-term-2}
N\left(|\ti A|^2-\frac12 \tilde{H}^2\right)= & 2(h^S_{\nu \nu} - 4h^S_{22})\left(|\ti A|^2-\frac12 \tilde{H}^2\right) - 2h^S_{22} h_{11} (h_{22}-h_{11}) \\
& +2 (\nabla^S_\nu h^S_{22}) (h_{11}-h_{22}). \nonumber
\end{align}
\end{lemma}

\begin{proof}
For simplicity, we denote $V:=\nu_S^T$. Using Fermi coordinates near $\partial \Sigma$ and writing $V=V_1 \partial_1 + V_2 \partial_2$, we have
\[ |\nabla H-h^S_{\nu\nu}HV|^2= (\partial_1 H - h^S_{\nu \nu} H V_1)^2 + g^{22} (\partial_2 H - h^S_{\nu \nu} H V_2)^2 \]
and $|\nabla H-h^S_{\nu\nu}HV|^2=(\partial_2 H)^2$ at $\partial \Sigma$ since $V_1=1$, $V_2=0$, and $\partial_1 H=h^S_{\nu \nu} H$ by Lemma \ref{L:N-H}. Moreover, along $\partial \Sigma$ we have $\partial_1 g^{22} =-\partial_1 g_{22}=-2 h^S_{22}$ and $\partial_1 V_2 =0$. Therefore, putting all these together, we have
\begin{align*}
\frac{1}{2} \partial_1 |\nabla H-h^S_{\nu\nu}HV|^2 = & -h^S_{22} (\partial_2 H)^2 +(\partial_2 H)(\partial_1 \partial_2 H) \\
=& (h^S_{\nu \nu} -h^S_{22}) (\partial_2 H)^2 + H \partial_2 H \partial_2 h^S_{\nu \nu}.
\end{align*}
We then have (\ref{E:N-term-1}), noting that $\partial_2 h^S_{\nu \nu}=\nabla^S_2 h^S_{\nu \nu} +2 h_{22} h^S_{2 \nu}$.

For (\ref{E:N-term-2}), we compute using (\ref{E:N-|A|-tilde}) and Lemma \ref{L:A-tilde-1}, \ref{L:N-htilde} that
\begin{align*}
\frac{1}{2} N \left(|\ti A|^2-\frac12 \tilde{H}^2\right) =& h^S_{\nu\nu} \left(|\ti A|^2-\frac12 \tilde{H}^2\right)  + h^S_{22}(3Hh_{11}-2|\tilde{A}|^2-2 h^2_{11}) \\
& + (\nabla^S_\nu h^S_{22}) (h_{11}-h_{22}).
\end{align*} 
Along $\partial \Sigma$, we have $|\tilde{A}|^2=h^2_{11}+h^2_{22}$ and $H=h_{11}+h_{22}$ by Lemma \ref{L:A-tilde-1}. Therefore,
\begin{align}
\label{E:nice-trick}
3Hh_{11}-2|\tilde{A}|^2-2 h^2_{11} = & -2 (h_{22}-h_{11})^2 -h_{11} (h_{22}-h_{11}) \\
= & -4 \left(|\ti A|^2-\frac12  \tilde{H}^2\right)  -h_{11} (h_{22}-h_{11})  \nonumber
\end{align}
noting that $|\ti A|^2-\frac12  \tilde{H}^2 = \frac{1}{2} (h_{22}-h_{11})^2$.
\end{proof}

Next we have to compute the evolution equations of the terms in $g$. We first establish a lemma.

\begin{lemma}
\label{L:h^S-evolution}
We have the following evolution equation:
\begin{equation*}
(\partial_t -\Delta) h^S_{\nu \nu} =  2 |A|^2 h^S_{\nu \nu}  -4h_{pk} D_p h^S_{k \nu}-2 h_{pk} h_{p \ell} h^S_{k \ell} -D^2_{p,p} h^S_{\nu \nu}.
\end{equation*}
In particular, we have the bounds $\nabla h^S_{\nu \nu}=O(1+|A|)$ and $(\partial_t -\Delta) h^S_{\nu\nu}=O(1+|A|^2)$.
\end{lemma}

\begin{proof}
The calculation is similar to Proposition \ref{P:P-estimate}, and it is even simpler in this case since $h^S_{\nu \nu}$ is just a function. Choose any orthonormal geodesic coordinates $\partial_1,\partial_2$ centered at a point $x \in \Sigma$. We have
\[
\nabla_p h^S_{\nu \nu} = D_p h^S_{\nu \nu} +2h_{pk} h^S_{k\nu}.
\]
Differentiating again, using Codazzi equation, we have
\[
\nabla_q (D_p h^S_{\nu \nu} )=D^2_{q,p} h^S_{\nu \nu}  - h_{qp} D_\nu h^S_{\nu \nu} +2h_{qk} D_p h^S_{k \nu} ,
\]
\[
\nabla_q(h_{pk} h^S_{k\nu}) = (\nabla_k h_{pq}) h^S_{k\nu} + h_{pk} (D_q h^S_{k\nu}  -h_{qk} h^S_{\nu \nu} +h_{q \ell} h^S_{k\ell}).
\]
Adding up the terms and summing over $p,q$, we have
\begin{align*}
\Delta h^S_{\nu \nu}  =& -2 |A|^2 h^S_{\nu \nu} + 2(\nabla_k H) h^S_{k \nu} -HD_\nu h^S_{\nu \nu}  +4h_{pk} D_p h^S_{k \nu} \\
& +2 h_{pk} h_{p \ell} h^S_{k \ell} +D^2_{p,p} h^S_{\nu \nu}.
\end{align*}
On the other hand, computing the time derivative gives
\[
\partial_t h^S_{\nu \nu}  =2(\nabla_k H) h^S_{k \nu} - H D_\nu h^S_{\nu \nu}.
\]
Combining the last two equations yield the desired formula.
\end{proof}

Using the lemma above, we derive the following bounds on the evolution of the first term in $g$. Recall that we always have $H \geq 1$ and $|A|^2 \leq H^2$.

\begin{lemma}
\label{L:g-1st-term-evolution}
We have the following evolution equations:
\[ (\partial_t -\Delta) |\nabla H-h^S_{\nu\nu}H\nu^T_S|^2 \leq C(S) H^2 |\nabla A|^2 +C(S) H^4 -2  |\nabla(\nabla H-h^S_{\nu\nu}H \nu_S^T)|^2 \]
\[ (\partial_t -\Delta) \frac{|\nabla H-h^S_{\nu\nu}H\nu^T_S|^2}{H} \leq  C(S)H |\nabla A|^2 +C(S)H^3.\]
\end{lemma}

\begin{proof}
We write $V=\nu_S^T$ as before. From \cite[Lemma 9.6]{Edelen16}, we have
\begin{equation}
\label{E:V-bound}
\nabla V = O(1), \qquad \Delta V=O(H),
\end{equation}
\begin{equation}
\label{E:V-evolve}
\partial_t V_i = -H D_\nu V_i -H h_{ij} V_j -\partial_i H \langle V,\nu\rangle.
\end{equation}
Direct computation as in \cite[Lemma 9.6]{Edelen16} together with Lemma \ref{L:h^S-evolution} gives
\begin{align*}
\frac{1}{2} \Delta |\nabla H-h^S_{\nu\nu}HV|^2 = & |\nabla(\nabla H-h^S_{\nu\nu}HV)|^2 + (\nabla_i H - h^S_{\nu \nu} H V_i) \cdot \\
& \Big( \nabla_i \Delta H + \nabla_jH (Hh_{ij}-h_{ik}h_{kj})   \\
&  -(\Delta H) h^S_{\nu \nu} V_i - (\Delta h^S_{\nu \nu}) H V_i + O(H^2 +H|\nabla H|)   \Big)
\end{align*}
and
\begin{align*}
\frac{1}{2} \partial_t  |\nabla H-h^S_{\nu\nu}HV|^2  =& (\nabla_i H - h^S_{\nu \nu} H V_i) \cdot \Big( H h_{ij} \nabla_j H + \nabla_i(\Delta H +|A|^2 H)    \\
&  - (\Delta H+|A|^2 H) h^S_{\nu \nu} V_i - (\partial_t h^S_{\nu \nu}) H V_i + O(H^2 +|\nabla H|)    \Big).
\end{align*}
Combining the two equations above, we obtain
\begin{align*}
(\partial_t -\Delta)  |\nabla H-h^S_{\nu\nu}HV|^2 = & -2  |\nabla(\nabla H-h^S_{\nu\nu}H \nu_S^T)|^2  +2 (\nabla_i H - h^S_{\nu \nu} H V_i) \cdot \\
& \Big( \nabla_i(|A|^2 H) -|A|^2 H h^S_{\nu \nu} V_i + h_{ik}h_{kj} \nabla_j H      \\
& -H V_i (\partial_t -\Delta) h^S_{\nu\nu}  + O(H^2 +H|\nabla H|)  \Big)
\end{align*}
from which the first estimate follows. The first estimate then implies the second one as in \cite[Lemma 9.6]{Edelen16}.
\end{proof}

\begin{lemma}
\label{L:g-2nd-term-evolution}
We have the following evolution equations:
$$
(\partial_t-\Delta)H^3\ge-6H|\nabla H|^2+\frac{3}{2}H^5,
$$
$$
(\partial_t-\Delta) \left(H\left(|\tilde A|^2-\frac12\tilde H^2\right)\right)\le-\frac{1}{3} H|\nabla A|^2+C(S)|\nabla A|^2+C(S,\Sigma_0)H^{5-\sigma}.
$$
\end{lemma}

\begin{proof}
The first inequality follows from \cite[Lemma 6.5]{Huisken84} and Cauchy-Schwarz inequality $|A|^2 \geq H^2/2$. From Lemma \ref{L:MCF}(v), (\ref{E:pinch-evolve}), (\ref{E:A-Atilde-bound}), (\ref{E:nablaA-Atilde-bound}) and the same calculations as in \cite[Lemma 6.5]{Huisken84}, we have
\begin{align*}
 (\partial_t-\Delta) \left(H\left(|\tilde{A}|^2-\frac12 \tilde{H}^2\right)\right) \le &3|A|^2 H  \left(|\tilde{A}|^2-\frac12 \tilde{H}^2\right) - 2H \left( |\nabla \tilde{A}|^2 - \frac{1}{2} |\nabla \tilde{H}|^2 \right) \\
 & +4 |\nabla H| |\nabla \tilde{A}| \sqrt{|\tilde{A}|^2-\frac12 \tilde{H}^2} +O(H^4)\\
 \le &3|A|^2 H  \left(|\tilde{A}|^2-\frac12 \tilde{H}^2\right) - 2H \left( |\nabla A|^2 - \frac{1}{2} |\nabla H|^2 \right) \\
 & +4 |\nabla H| |\nabla \tilde{A}| \sqrt{|\tilde{A}|^2-\frac12 \tilde{H}^2} +O(|\nabla A|^2+H^4)
 \end{align*}
 Applying the pinching estimate of Theorem \ref{T:pinching-A-tilde}, \cite[Lemma 2.2 (ii)]{Huisken84} and using (\ref{E:A-Atilde-bound}), (\ref{E:nablaA-Atilde-bound}) again, together with Peter-Paul inequality, we have
\begin{align*}
(\partial_t-\Delta) \left(H\left(|\tilde{A}|^2-\frac12 \tilde{H}^2\right)\right)\le & C(S,\tilde{C}_0) |A|^2 H^{3-\sigma} - \frac{2}{3} H |\nabla A|^2 \\
& + \frac{1}{3} H |\nabla A|^2 +C(S) |\nabla A|^2 +C(S,\tilde{C}_0) H^4 \\
\leq & - \frac{1}{3} H |\nabla A|^2 +C(S) |\nabla A|^2 + C(S,\Sigma_0) H^{5-\sigma},
\end{align*}
which proves our desired inequality.
\end{proof}

We are now ready to give the proof of Theorem \ref{T:H-gradient-bound}. Recall that we always have $H \geq 1$ and $|A|^2 \leq H^2$.

\begin{proof}[Proof of Theorem \ref{T:H-gradient-bound}]
The proof is again a maximum principle argument. We first analyse the boundary derivatives of $g$ term by term. By (\ref{E:N-term-1}), using triangle inequality and Peter-Paul inequality, we have
\begin{align*} 
N \left( \frac{|\nabla H - h^S_{\nu \nu} H\nu_S^T|^2}{H}  \right) 
\le& C(S) \frac{|\nabla H - h^S_{\nu \nu} H\nu_S^T|^2}{H} +C(S)H|\nabla H|\\
\le&C(S) \frac{|\nabla H - h^S_{\nu \nu} H\nu_S^T|^2}{H}+C(S)H|\nabla H - h^S_{\nu \nu} H\nu_S^T|+C(S)H^2\\
\le&C(S) \frac{|\nabla H - h^S_{\nu \nu} H\nu_S^T|^2}{H}+\frac{1}{4}H^3 + C(S).
\end{align*}
Next, using (\ref{E:N-term-2}), Lemma \ref{L:N-H}, Theorem \ref{T:pinching-A-tilde} and Peter-Paul inequality, we have
\begin{align*}
N\left( bH \Big( |\tilde{A}|^2-\frac{1}{2} \tilde{H}^2 \Big) \right) 
\le&bC(S)H\Big( |\tilde{A}|^2-\frac{1}{2} \tilde{H}^2 \Big) +bC(S)H^2\sqrt{\Big( |\tilde{A}|^2-\frac{1}{2} \tilde{H}^2 \Big) }\\
\le&bC(S,\Sigma_0)H^{3-\sigma} \\
\le & \frac{1}{4} H^3 +C(S,\Sigma_0,b).
\end{align*}

Next, using \eqref{E:N-|A|-tilde} and \eqref{E:zeta-boundary} (and that $h^S \geq 0$), together with Peter-Paul, we have
\[
N(ba |\tilde{A}|^2 -\zeta H^3) \le ba C(S)|\ti A|^2  - H^3 \leq -\frac{3}{4} H^3  +ba C(S)
\]
Combining all the above estimates, we obtain 
\begin{align*}
Ng\le&C(S) \frac{|\nabla H - h^S_{\nu \nu} H\nu_S^T|^2}{H}-\frac{1}{4} H^3 + C(S,\Sigma_0,a,b)\\
\le & C(S)g +\left(\eta C(S) -\frac{1}{4}\right)H^3- c C(S) + C(S,\Sigma_0,a,b).
\end{align*}
Hence, by choosing $\eta=\eta(S)>0$ sufficiently small and $c=c(S,\Sigma_0,a,b)>0$ sufficiently large, we then have $Ng \leq C(S) g$.
This implies for $d=d(S)>0$ sufficiently large, we have
$$
N(e^{-d\rho} g) \le -dg+C(S)g<0.
$$
Hence $e^{-d\rho} g$ cannot attain a maximum on the boundary $\partial \Sigma$ for these choices of the constants $c$ and $d$.

Now we proceed to study the evolution equation of $g$ term by term. 
First, from Lemma \ref{L:g-1st-term-evolution} we have
\[ (\partial_t -\Delta) \frac{|\nabla H-h^S_{\nu\nu}H\nu^T_S|^2}{H} \leq  C(S)H |\nabla A|^2 +C(S) H^3.\]
Next, Lemma \ref{L:g-2nd-term-evolution} implies
\begin{align*}
(\partial_t-\Delta) \left(bH\left(|\tilde A|^2-\frac12\tilde H^2\right)\right)
\le -\frac{b}{3} H|\nabla A|^2+bC(S)|\nabla A|^2+b C(S,\Sigma_0) H^{5-\sigma}.
\end{align*}
On the other hand, (\ref{E:|Atilde|^2-evolve}), (\ref{E:A-Atilde-bound}) and (\ref{E:nablaA-Atilde-bound}) imply
\begin{align*}
(\partial_t-\Delta)(ba|\tilde A|^2) \le & -2ba|\nabla \tilde{A}|^2 +2ba |A|^2 |\tilde{A}|^2 +ba C(S) H^2\\
\leq & -2ba|\nabla A|^2+ba C(S) H |\nabla A|+ba C(S) H^4 \\
\leq & - ba |\nabla A|^2 +ba C(S) H^4,
\end{align*}
where we have used Cauchy-Schwarz in the last inequality. Using Lemma \ref{L:MCF}(v), \eqref{E:zeta-interior} and Lemma \ref{L:g-2nd-term-evolution}, we have
\begin{align*}
(\partial_t-\Delta)(-\zeta H^3)\le& \zeta \left(6H|\nabla A|^2-\frac{3}{2}H^5\right) +\frac{55e^2}{\eta} H^3 +30e^2H^2 |\nabla H|\\
\le& -\frac{3e^{-2}}{2} \eta  H^5 + 6 e^2(\eta+5)  H|\nabla A|^2+\frac{55 e^2}{\eta} H^3.
\end{align*}
Combining all the above inequalities, we obtain (recall that $\eta < 1$)
\begin{align*}
(\partial_t-\Delta)g\le  \left( -\frac{b}{3} +C(S)  \right) &H |\nabla A|^2  + b (C(S)-a) |\nabla A|^2 \\
 &  -\frac{3e^{-2}}{2} \eta  H^5+ C(S,\Sigma_0,a,b,\eta) H^{5-\sigma}.
\end{align*}
By choosing $a=a(S)$ and $b=b(S)$ sufficiently large, using Peter-Paul inequality, we arrive at
\begin{equation}
\label{E:g-bound}
(\partial_t-\Delta)g\le C(S,\Sigma_0,a,b,\eta).
\end{equation}
We now consider the function $\varphi:=e^{-d\rho-ft} g$. Note that $\varphi$ cannot attain a boundary maximum. Moreover, we compute using (\ref{E:g-bound}), \eqref{E:rho-estimate} and \eqref{E:rho-estimate-2} that
\begin{align*}
(\partial_t-\Delta)\varphi =&  -f \varphi +e^{-d\rho-ft}(\partial_t-\Delta)g + e^{-ft} g (\partial_t - \Delta) (e^{-d \rho}) - 2 e^{-ft} \nabla e^{-d\rho} \cdot \nabla g\\
\le &-f \varphi +C(S,\Sigma_0,a,b,d,\eta) +C(d,\eta) \varphi  - 2 e^{-ft} \nabla e^{-d\rho} \cdot \nabla g.
\end{align*}
Suppose we are looking at a spatial interior maximum of $\varphi$. Then we have $\nabla\varphi=0$ at this point, which implies $\nabla g =dg \nabla \rho$, hence the gradient term above can be estimated using \eqref{E:rho-estimate}
$$
- 2 e^{-ft} \nabla e^{-d\rho} \cdot \nabla g \leq C(d) \varphi.
$$
Putting this back to the inequality above, we have
\[ (\partial_t-\Delta)\varphi \leq (-f+C(d,\eta)) \varphi + C(S,\Sigma_0,a,b,d,\eta).\]
By choosing $f=f(d,\eta)>0$ sufficiently large, we obtain that the maximum of $\varphi$ can at most increase linearly with time. Finally, observe that the constants $a,b,d,f$ only depend on $S$, $\eta$ has to be small depending only on $S$ and that $c=c(S,\Sigma_0,a,b)$ large enough. Moreover, $T=T(S,\Sigma_0)$. Therefore, we have
$$
\varphi(x,t)\le C(S,\Sigma_0,\eta)
$$
holds on $\Sigma \times [0,T)$. Since $T<\infty$ and $\rho$ is bounded by (\ref{E:rho-estimate}), we deduce that 
$$
g(x,t)\le C(S,\Sigma_0,\eta)
$$
holds on $\Sigma \times [0,T)$. Dropping the nonnegative terms in $g$, we have
$$
|\nabla H-h^S_{\nu\nu}H\nu_S^T|^2\le  \zeta H^4 +C(S,\Sigma_0,\eta)H.
$$
Thus the result follows from the bound on $\zeta$ in (\ref{E:zeta-interior}), the triangle inequality and Peter-Paul inequality.
\end{proof}

\bibliographystyle{plain}
\bibliography{references}

\begin{thebibliography}{10}

\bibitem{Anderson08}
Michael~T. Anderson.
\newblock On boundary value problems for {E}instein metrics.
\newblock {\em Geom. Topol.}, 12(4):2009--2045, 2008.

\bibitem{Anderson12}
Michael~T. Anderson.
\newblock Boundary value problems for metrics on 3-manifolds.
\newblock In {\em Metric and differential geometry}, volume 297 of {\em Progr.
  Math.}, pages 3--17. Birkh\"{a}user/Springer, Basel, 2012.

\bibitem{Andrews-Langford-McCoy13}
Ben Andrews, Mat Langford, and James McCoy.
\newblock Non-collapsing in fully non-linear curvature flows.
\newblock {\em Ann. Inst. H. Poincar\'{e} Anal. Non Lin\'{e}aire},
  30(1):23--32, 2013.

\bibitem{Brendle-Schoen09}
Simon Brendle and Richard Schoen.
\newblock Manifolds with {$1/4$}-pinched curvature are space forms.
\newblock {\em J. Amer. Math. Soc.}, 22(1):287--307, 2009.

\bibitem{Buckland05}
John~A. Buckland.
\newblock Mean curvature flow with free boundary on smooth hypersurfaces.
\newblock {\em J. Reine Angew. Math.}, 586:71--90, 2005.

\bibitem{Edelen16}
Nick Edelen.
\newblock Convexity estimates for mean curvature flow with free boundary.
\newblock {\em Adv. Math.}, 294:1--36, 2016.

\bibitem{Edelen18}
Nick Edelen.
\newblock The free-boundary brakke flow.
\newblock {\em J. Reine Angew. Math.}, 2018.

\bibitem{Edelen-Haslhofer-Ivaki-Zhu19}
Nick Edelen, Robert Haslhofer, Mohammad Ivaki, and Jonathan Zhu.
\newblock Mean convex mean curvature flow with free boundary.
\newblock arXiv:1911.01186.

\bibitem{Evans-Lambert-Wood}
Christopher Evans, Ben Lambert, and Albert Wood.
\newblock Lagrangian mean curvature flow with boundary.
\newblock arXiv:1911.04977.

\bibitem{Ghomi-Xiong19}
Mohammad Ghomi and Changwei Xiong.
\newblock Nonnegatively curved hypersurfaces with free boundary on a sphere.
\newblock {\em Calc. Var. Partial Differential Equations}, 58(3):Art. 94, 20,
  2019.

\bibitem{Gianniotis16b}
Panagiotis Gianniotis.
\newblock Boundary estimates for the {R}icci flow.
\newblock {\em Calc. Var. Partial Differential Equations}, 55(1):Art. 9, 21,
  2016.

\bibitem{Gianniotis16a}
Panagiotis Gianniotis.
\newblock The {R}icci flow on manifolds with boundary.
\newblock {\em J. Differential Geom.}, 104(2):291--324, 2016.

\bibitem{Giga-Sato93}
Yoshikazu Giga and Moto-Hiko Sato.
\newblock Neumann problem for singular degenerate parabolic equations.
\newblock {\em Differential Integral Equations}, 6(6):1217--1230, 1993.

\bibitem{Gilbarg-Trudinger}
David Gilbarg and Neil~S. Trudinger.
\newblock {\em Elliptic partial differential equations of second order}.
\newblock Classics in Mathematics. Springer-Verlag, Berlin, 2001.
\newblock Reprint of the 1998 edition.

\bibitem{Guo}
Siao-Hao Guo.
\newblock Extension of two-dimensional mean curvature flow with free boundary.
\newblock arXiv:1807.02922.

\bibitem{Hamilton82}
Richard~S. Hamilton.
\newblock Three-manifolds with positive {R}icci curvature.
\newblock {\em J. Differential Geom.}, 17(2):255--306, 1982.

\bibitem{Huisken84}
Gerhard Huisken.
\newblock Flow by mean curvature of convex surfaces into spheres.
\newblock {\em J. Differential Geom.}, 20(1):237--266, 1984.

\bibitem{Huisken86}
Gerhard Huisken.
\newblock Contracting convex hypersurfaces in {R}iemannian manifolds by their
  mean curvature.
\newblock {\em Invent. Math.}, 84(3):463--480, 1986.

\bibitem{Huisken89}
Gerhard Huisken.
\newblock Nonparametric mean curvature evolution with boundary conditions.
\newblock {\em J. Differential Equations}, 77(2):369--378, 1989.

\bibitem{Huisken-Ilmanen01}
Gerhard Huisken and Tom Ilmanen.
\newblock The inverse mean curvature flow and the {R}iemannian {P}enrose
  inequality.
\newblock {\em J. Differential Geom.}, 59(3):353--437, 2001.

\bibitem{Huisken-Sinestrari99b}
Gerhard Huisken and Carlo Sinestrari.
\newblock Convexity estimates for mean curvature flow and singularities of mean
  convex surfaces.
\newblock {\em Acta Math.}, 183(1):45--70, 1999.

\bibitem{Huisken-Sinestrari99a}
Gerhard Huisken and Carlo Sinestrari.
\newblock Mean curvature flow singularities for mean convex surfaces.
\newblock {\em Calc. Var. Partial Differential Equations}, 8(1):1--14, 1999.

\bibitem{Koeller12}
Amos~N. Koeller.
\newblock Regularity of mean curvature flows with {N}eumann free boundary
  conditions.
\newblock {\em Calc. Var. Partial Differential Equations}, 43(1-2):265--309,
  2012.

\bibitem{Lambert14}
Ben Lambert.
\newblock The perpendicular {N}eumann problem for mean curvature flow with a
  timelike cone boundary condition.
\newblock {\em Trans. Amer. Math. Soc.}, 366(7):3373--3388, 2014.

\bibitem{Li-Wang19}
Haozhao Li and Bing Wang.
\newblock The extension problem of the mean curvature flow ({I}).
\newblock {\em Invent. Math.}, 218(3):721--777, 2019.

\bibitem{Marques05}
Fernando~C. Marques.
\newblock Existence results for the {Y}amabe problem on manifolds with
  boundary.
\newblock {\em Indiana Univ. Math. J.}, 54(6):1599--1620, 2005.

\bibitem{Mizuno-Tonegawa15}
Masashi Mizuno and Yoshihiro Tonegawa.
\newblock Convergence of the {A}llen-{C}ahn equation with {N}eumann boundary
  conditions.
\newblock {\em SIAM J. Math. Anal.}, 47(3):1906--1932, 2015.

\bibitem{Perelman1}
Grisha Perelman.
\newblock The entropy formula for the ricci flow and its geometric
  applications.
\newblock arXiv:math/0211159.

\bibitem{Perelman3}
Grisha Perelman.
\newblock Finite extinction time for the solutions to the ricci flow on certain
  three-manifolds.
\newblock arXiv:math/0307245.

\bibitem{Perelman2}
Grisha Perelman.
\newblock Ricci flow with surgery on three-manifolds.
\newblock arXiv:math/0303109.

\bibitem{Samelson69}
Hans Samelson.
\newblock Orientability of hypersurfaces in {$R^{n}$}.
\newblock {\em Proc. Amer. Math. Soc.}, 22:301--302, 1969.

\bibitem{Stahl96b}
Axel Stahl.
\newblock Convergence of solutions to the mean curvature flow with a {N}eumann
  boundary condition.
\newblock {\em Calc. Var. Partial Differential Equations}, 4(5):421--441, 1996.

\bibitem{Stahl96a}
Axel Stahl.
\newblock Regularity estimates for solutions to the mean curvature flow with a
  {N}eumann boundary condition.
\newblock {\em Calc. Var. Partial Differential Equations}, 4(4):385--407, 1996.

\bibitem{Wheeler14b}
Valentina~Mira Wheeler.
\newblock Mean curvature flow of entire graphs in a half-space with a free
  boundary.
\newblock {\em J. Reine Angew. Math.}, 690:115--131, 2014.

\bibitem{Wheeler14}
Valentina-Mira Wheeler.
\newblock Non-parametric radially symmetric mean curvature flow with a free
  boundary.
\newblock {\em Math. Z.}, 276(1-2):281--298, 2014.

\bibitem{White}
Brian White.
\newblock Mean curvature flow with boundary.
\newblock arXiv:1901.03008.

\bibitem{White00}
Brian White.
\newblock The size of the singular set in mean curvature flow of mean-convex
  sets.
\newblock {\em J. Amer. Math. Soc.}, 13(3):665--695, 2000.

\bibitem{White03}
Brian White.
\newblock The nature of singularities in mean curvature flow of mean-convex
  sets.
\newblock {\em J. Amer. Math. Soc.}, 16(1):123--138, 2003.

\bibitem{White15}
Brian White.
\newblock Subsequent singularities in mean-convex mean curvature flow.
\newblock {\em Calc. Var. Partial Differential Equations}, 54(2):1457--1468,
  2015.

\end{thebibliography}

\end{document}